\documentclass[reqno]{amsart}
\usepackage{xcolor}
\usepackage{amsthm,amsmath,amssymb,latexsym,soul,cite,mathrsfs}
\usepackage{enumitem}
\usepackage{color,enumitem}
\usepackage[colorlinks=true,urlcolor=blue,
citecolor=red,linkcolor=blue,linktocpage,pdfpagelabels,
bookmarksnumbered,bookmarksopen]{hyperref}
\usepackage[english]{babel}

\usepackage{graphicx}

\usepackage[left=2.7cm,right=2.7cm,top=2.9cm,bottom=2.9cm]{geometry}

\usepackage[hyperpageref]{backref}

\numberwithin{equation}{section}

\newtheorem{theorem}{Theorem}[section]
\theoremstyle{plain}
\newtheorem{theoremletter}{Theorem}

\newtheorem{lemma}[theorem]{Lemma}
\newtheorem{lemmaletter}[theoremletter]{Lemma}
\newtheorem{corollary}[theorem]{Corollary}
\newtheorem{proposition}[theorem]{Proposition}

\newtheorem{remark}[theorem]{Remark}

\newcommand{\dx}{\,\mathrm{d}x}
\newcommand{\dy}{\,\mathrm{d}y}
\newcommand{\dz}{\,\mathrm{d}z}

\newcommand{\dtau}{\,\mathrm{d}\tau}

\newcommand{\dxdy}{\,\mathrm{d}x\mathrm{d}y}
\newcommand{\drho}{\,\mathrm{d}\varrho}

\DeclareMathOperator{\supp}{supp}

\newcommand{\loca}{\operatorname{loc}}

\usepackage[norefs,nocites]{refcheck}

\title[]{Supercritical Schr\"odinger equations involving integro-differential operators and vanishing potentials}

\author[R.C.~Duarte]{Ronaldo C. Duarte}
\address{Department of Mathematics,
	Federal University of Rio Grande do Norte
	59078-970, Natal-RN, Brazil}
\email{ronaldo.cesar.duarte@ufrn.br}

\author[D.~Ferraz]{Diego Ferraz}
\address{Department of Mathematics,
	Federal University of Rio Grande do Norte
	59078-970, Natal-RN, Brazil}
\email{diego.ferraz.br@gmail.com}

\keywords{Integro-differential operators, Vanishing potentials, Supercritical growth}

\subjclass[2020]{Primary: 35J60, 35A15; Secondary: 47G20, 35R11.}
\date{\today}
\begin{document}
	
		\maketitle

	\begin{abstract}
		This paper is devoted to the study of the existence of positive and bounded solutions for a Schr\"odinger type equation defined on the entire Euclidean space, involving a general integro-differential operator. We consider the case where the potential is nonnegative and vanishes at infinity with a nonlinearity exhibiting critical or supercritical growth in the Sobolev sense. To overcome the lack of compactness and the difficulties imposed by the general structure of the nonlinearity, we employ variational methods combined with a penalization technique. Unlike the classical fractional Laplacian framework, where specific regularity results, decay estimates, and the $s$-harmonic extension are available, our approach relies on a weak Maximum Principle combined with the construction of a supersolution based on the truncated fundamental solution of the fractional Laplacian to control the asymptotic behavior of the solutions.  We prove that, for sufficiently small perturbation parameters and under suitable decay conditions on the potential, the equation admits a nontrivial solution.
	\end{abstract}

	\section{Introduction}
	
	In this paper, we study the existence of solutions to the Schr\"{o}dinger equation
	\begin{equation}\label{P}\tag{$P$}
		 \mathcal{L}_{K_s} u + V(x)u = f(u) + \lambda|u|^{q-2}u \quad \text{in } \mathbb{R}^N,
	\end{equation}
	where $N > 2s$ and $0<s<1,$ involving the integro-differential operator $\mathcal{L}_{K}u$ given by
	\begin{equation}\label{equ1}
		 \mathcal{L}_{K_s}u(x)= \lim _{\varepsilon \rightarrow 0^+} \int _{\mathbb{R}^N \setminus B_{\varepsilon} (x)} (u(x) - u(y))K (x-y) \dy.
	\end{equation}
	Here, $V$ is a nonnegative potential vanishing at infinity, $\lambda > 0$, and $f$ is a continuous function exhibiting subcritical growth at infinity and critical growth at the origin in the fractional Sobolev sense. The critical or supercritical case $q \geq 2^{\ast}_s = 2N/(N-2s)$ is considered. The precise hypotheses are stated below.
	
	As highlighted in the recent monograph \cite{zbMATH07813655}, a primary motivation for investigating operators of the form \eqref{equ1} lies in their intrinsic connection to the theory of stochastic processes. Specifically, the operator $\mathcal{L}_{K_s}$ arises as the infinitesimal generator of a symmetric L\'{e}vy process. While the classical Laplacian is associated with Brownian motion (describing continuous paths), nonlocal operators correspond to processes allowing for jump discontinuities (L\'{e}vy flights). In this probabilistic framework, the kernel $K_s(x-y)$ represents the density of the jump measure, determining the likelihood of a particle jumping from $y$ to $x$. Beyond this probabilistic foundation, nonlocal operators arise naturally in a wide range of mathematical models and real-world applications. Notable examples include phase transitions (\cite{phase1,phase2,phase3}), obstacle problems (\cite{obstaculo1,obstaculo2}), and optimization strategies (\cite{otimizacao1}). For a comprehensive overview of these applications and further references, we refer the reader to \cite{outros1, outros2, outros3, outros4, dpv} and the references therein.

In the context of the classical Laplacian, C. O. Alves and M. A. S. Souto \cite{alv} investigated the Schr\"{o}dinger equation $-\Delta u + V(x) u = f(u)$ in $\mathbb{R}^N$, considering a nonnegative potential vanishing at infinity and a nonlinearity subject to growth conditions closely related to ours. To overcome the lack of compactness, they introduced a framework for potentials $V$ satisfying a specific decay behavior away from the origin. Precisely, they assumed that there exist constants $\Lambda > 0$ and $R > 1$ such that
\begin{equation}\label{alvv}
	\frac{1}{R^4}\inf _{|x| \geq R} |x|^4 V(x) \geq \Lambda.
\end{equation}
Since the publication of \cite{alv}, the study of elliptic equations with potentials satisfying conditions like \eqref{alvv} has attracted significant attention. For related results, we refer the reader to \cite{alv3,Ambro1,Bon1,uber1} and the references therein.
	
	On the other hand, a primary motivation for the study of the integro-differential operators defined in \eqref{equ1} arises from the particular case where the kernel is given by $K_s(z) = C_{N,s} |z|^{-(N+2s)},$ and $C_{N,s}$ is a suitable normalization constant. In this setting, $\mathcal{L}_{K_s}$ corresponds to the well-known fractional Laplacian $(-\Delta )^s$ (see \cite{dpv}), and Eq. \eqref{P} reduces to
	\begin{equation}\label{fp}
		(-\Delta )^s u + V(x)u = f(u) + \lambda|u|^{q-2}u, \quad \text{in } \mathbb{R}^N.
	\end{equation}
	The literature concerning elliptic problems involving the fractional Laplacian and variational methods is vast. We highlight, for instance, the works \cite{dpv,ambrosio2016,ambrosio1,zbMATH07061025,bisci1,caf,uber1,fel}. Specifically in \cite{zbMATH07061025}, Q. Li, K. Teng, X. Wu and W. Wang investigated Eq. \eqref{fp} assuming that $V$ is coercive and bounded away from zero (i.e., $\lim _{|x| \rightarrow \infty } V(x) = +\infty$ and $\inf_{\mathbb{R}^N} V > 0$), and that $f$ satisfies the following growth condition:
	\begin{enumerate}[label=($f_I$),ref=$(f_I)$]
		\item There exists $p \in (2, 2^\ast _s)$ such that $|f(t)| \leq C(1 + |t|^{p-1})$ and $\lim _{t \rightarrow 0} f(t)t^{-1} = 0.$
	\end{enumerate}
	This hypothesis, combined with other structural conditions in \cite{zbMATH07061025}, implies that $f$ behaves like a power function of the form $t \mapsto |t|^{p-2}t$.
		
	Our study is strongly motivated by the fractional Laplacian framework, particularly the work of J. A. Cardoso, D. S. dos Prazeres and U. B. Severo \cite{uber1}. Their paper presents a refined analysis that synthesizes the methods of \cite{alv} and \cite{zbMATH07061025} to address vanishing potentials satisfying \eqref{alvv}. Our main objective is to extend and generalize these results. Roughly speaking, the strategy employed in \cite{uber1} relies on an auxiliary problem defined by a suitable truncation of the nonlinearity $f(u) + \lambda |u|^{q-2}u$. A crucial step in their argument is establishing that the solution to the auxiliary problem decays to zero at infinity. In \cite{uber1}, this is achieved by proving a suitable uniform bound (via the Moser iteration technique) for the solutions of the auxiliary problem and  exploiting the specific regularity theory available for the fractional Laplacian. Once the decay is established, they use a Maximum Principle for the fractional Laplacian to show that the auxiliary solution solves the original equation. 
	
	However, extending this approach to the general integro-differential operators defined in \eqref{equ1} presents significant challenges. We emphasize that the powerful tools associated with the fractional Laplacian, such as the $s$-harmonic extension developed in \cite{caf}, are not available in this general context. Consequently, the regularity results used in \cite{uber1} cannot be directly applied, requiring new estimates to control the asymptotic behavior of the solution. To overcome these obstacles, we adopt an alternative strategy grounded in a recently developed tool: a weak Maximum Principle in $\mathbb{R}^N$ for operators of the form $u \mapsto \mathcal{L}_{K_s}u + a(x)u$ (with $a \in L^1_{\loca}(\mathbb{R}^N)$ and $a(x) \geq 0$), as established in \cite{duarte2}. This is combined with a refined analysis of the operator $\mathcal{L}_{K_s}$ acting on the truncated fundamental solution of the fractional Laplacian $\Gamma_\ast(x) = \min \{ R^{-(N-2s)} , |x|^{-(N-2s)} \}$. Specifically, we demonstrate that the results of \cite{uber1} can be extended to this setting provided $K_s$ is comparable to the standard fractional Laplacian kernel $|x|^{-(N+2s)}$ and that $\Gamma_\ast$ serves as a weak supersolution for $\mathcal{L}_{K_s}$ away from the origin (see hypotheses \ref{K_quatro} and \ref{K_cinco}). By combining these results, we successfully obtain the necessary estimates to generalize the main theorem of \cite{uber1} to our broader setting. In particular, the proof we provide for the uniform bound of the auxiliary solutions differs from \cite{uber1}, as the generality of our operator necessitates a distinct approach; in fact, the structural limitations of our general framework prevent the direct application of most arguments from \cite{uber1}, even the purely variational ones. Since we were able to overcome the obstacles imposed by the lack of regularity by relying strictly on variational arguments and weak comparison principles, our methodology can be seen as the natural extension of the well-established framework introduced by C. O. Alves and M. A. S. Souto \cite{alv} to the general non-local setting.
	
Before stating our main result, we point out that related problems involving integro-differential operators have been studied in \cite{gu1, iane1, iane2} under distinct structural conditions. For a comprehensive survey of this subject, we refer to \cite{bisci1,ser,zbMATH07813655}.
	
	\subsection{Hypotheses and main result} We assume that $f\in C(\mathbb R,\mathbb R)$ is a nonzero function satisfying $f(t)\geq0$ for all $t>0,$ and $f(t)=0$ for all $t\leq 0. $ Furthermore, we impose the following hypotheses:
	\begin{enumerate}[label=($f_1$),ref=$(f_1)$]
		\item \label{f_um}$\displaystyle \limsup_{t \rightarrow {0^{+} } }\frac{f(t)}{t^{ 2^{\ast}_s -1}}< + \infty$ and $\displaystyle \limsup_{t \rightarrow \infty}\frac{f(t)}{t^{p-1}}<+\infty,$ for some $p \in (2,2^{\ast}_s);$
	\end{enumerate}
	\begin{enumerate}[label=($f_2$),ref=$(f_2)$]
		\item\label{f_dois}There is $\theta \in (2,p)$ such that $\theta F(t)\leq tf(t),$ for all $t\in\mathbb{R},$ where $F(t) = \int _0 ^t f(\tau )\dtau;$
	\end{enumerate}
	Regarding the potential $V$, we assume the following conditions:
	\begin{enumerate}[label=($V_1$),ref=$(V_1)$]
		\item\label{V_um}$V \in L^{\infty}_{\loca}(\mathbb{R}^N)$ and $V(x)\geq 0$, for almost every (a.e.) $x \in \mathbb{R}^{N};$
	\end{enumerate}
	\begin{enumerate}[label=($V_2$),ref=$(V_2)$]
		\item\label{V_dois}There are $R>1$ and $\Lambda>0$ such that
		\begin{equation*}
			\frac{1}{R^{4s } }\inf_{|x|\geq R}V(x)|x|^{4s}\geq \Lambda;
		\end{equation*}
	\end{enumerate}
	The kernel $K_s$ is a positive measurable function on $\mathbb{R}^N$ satisfying:
		\begin{enumerate}[label=($K_1$),ref=$(K_1)$]
			\item\label{K_um} $K_s(x)=K_s(-x),$ for all $x \in \mathbb{R}^{N}$;
		\end{enumerate}
		\begin{enumerate}[label=($K_2$),ref=$(K_2)$]
			\item\label{K_dois}There is $\mathcal{C}_1 >0$ such that $K_s(x) \geq \mathcal{C}_1 |x|^{-(N+2s)}$ a.e. in $\mathbb{R}^{N}$;
		\end{enumerate}
		\begin{enumerate}[label=($K_3$),ref=$(K_3)$]
			\item\label{K_tres}$\gamma K_s \in L^{1}(\mathbb{R}^{N})$, where $\gamma(x)=\min\left\{|x|^{2},1\right\}$;
		\end{enumerate}
		Hypotheses \ref{K_um}--\ref{K_tres} provide the standard structural framework to study semilinear elliptic equations like \eqref{P} via variational methods. In addition, we assume:
		\begin{enumerate}[label=($K_4$),ref=$(K_4)$]
			\item\label{K_quatro} There is $\mathcal{C}_2 >0$ such that $ K_s(x) \leq  \mathcal{C}_2|x|^{-(N+2s)}, $ a.e. in $\mathbb{R}^{N};$
		\end{enumerate}
		\begin{enumerate}[label=($K_5$),ref=$(K_5)$]
			\item\label{K_cinco} The operator $\mathcal{L}_{K_s}$ admits the function $\Gamma_\ast(x) = \min\{ |x|^{-(N-2s)}, R^{-(N-2s)} \}$ as a supersolution in $\mathbb{R}^N \setminus \overline{B}_R$ in the weak sense.
		\end{enumerate}
In contrast, \ref{K_quatro} and \ref{K_cinco} are employed exclusively in Section \ref{s_proof} to establish our main result, serving specifically to control the asymptotic behavior of the auxiliary solution for a suitable choice of parameters.
Under the above hypotheses, our main result is stated as follows:
\begin{theorem}\label{main}
There are $\lambda_0>0$ and $\Lambda^{\ast}>0$ such that Eq. \eqref{P} has a positive bounded solution in $\mathbb{R}^N$, whenever $0<\lambda <\lambda_0$ and $\Lambda>\Lambda^{\ast}.$
\end{theorem}
Theorem \ref{main} guarantees the existence of a positive bounded solution in a general integro-differential framework. To the best of our knowledge, this is the first result in this direction for such a class of operators and potentials. To derive this solution, we adapt the strategy of \cite{uber1}, which synthesizes the techniques from \cite{alv} and \cite{zbMATH07061025}. In particular, our approach relies on the penalization method introduced by del Pino and Felmer \cite{fel}, which consists of modifying the nonlinearity to define a suitable auxiliary problem.
	\subsection{Remarks on the assumptions} Before we proceed, some comments on our hypotheses are necessary.
	\begin{enumerate}[label=\bf \roman*):]
		\item Our hypotheses allow for smooth perturbations of the classical fractional Laplacian kernel, provided the perturbative term vanishes at the origin. More precisely, let $a \in C^\infty(\mathbb{R}^N)$ be a radial cut-off function such that $0 \leq a(x) \leq 1$ in $\mathbb{R}^N$, with $a(x) = 0,$ for $|x| \leq R/2,$ and $a(x) = 1,$ for $|x| \geq R$. If we define the perturbed kernel $K_s(x) = (1 + a(x))|x|^{-(N+2s)}$, we show in Appendix \ref{s_app_ex} that $K_s$ satisfies conditions \ref{K_um}--\ref{K_cinco}.
		\item Hypotheses \ref{V_um}--\ref{V_dois} hold for the potential $V(x) = M (1+|x|^{r})^{-1},$ with $0 < r < 4s$ and $M>0$ sufficiently large.
		\item The function defined by $f(t) = t^{2^{\ast}_s - 1 }(1 + t^{2^{\ast}_s - p})^{-1},$ for $t>0$ (and zero otherwise), verifies \ref{f_um}--\ref{f_dois}.
		\item We point out that condition \ref{K_cinco} simply states that $\Gamma_\ast$ is a weak supersolution for the operator $\mathcal{L}_{K_s}$ outside the ball $\overline{B}_R$. That is, the inequality $\mathcal{L}_{K_s}(\Gamma_\ast) \geq 0$ in $\mathbb{R}^N \setminus \overline{B}_R$ means that
		\begin{equation*}
			\int_{\mathbb{R}^{N}}\int_{\mathbb{R}^{N}}(\Gamma_\ast(x)-\Gamma_\ast(y))(\phi(x)-\phi(y))K_s(x-y) \dxdy \geq 0,\quad \forall \, \phi \in C^\infty _0 (\mathbb{R}^N \setminus \overline{B}_R ),\ \phi \geq 0.
		\end{equation*}
		\item When $K_s(x) = C_{N,s}|x|^{-(N+2s)}$ is the standard fractional Laplacian kernel, hypothesis \ref{K_cinco} is readily verified. This follows from the fact that $\Gamma(x) = |x|^{-(N-2s)}$ is the fundamental solution of the fractional Laplacian (see \cite{MR5039745} and Appendix \ref{s_app_ex}).
		
	\end{enumerate}
	
	\subsection{Outline} The paper is organized as follows. In Section \ref{s_preli}, we present the variational framework and preliminary results required for our arguments. Section \ref{s_aux} is dedicated to the formulation of the auxiliary problem via the penalization method. In Section \ref{s_decay}, we proceed to establish the uniform boundedness of the auxiliary solutions. In Section \ref{s_proof}, we overcome the aforementioned lack of regularity by proving the asymptotic decay at infinity, which allows us to complete the proof of Theorem \ref{main}. In Appendix \ref{s_app_ex}, we construct a nontrivial example of a kernel satisfying our conditions, and in Appendix \ref{s_app_inequality}, we prove a technical inequality needed to establish the uniform bound obtained in Section \ref{s_decay}. Conditions \ref{f_um}--\ref{f_dois}, \ref{V_um}--\ref{V_dois}, and \ref{K_um}--\ref{K_tres} are assumed to hold throughout the text, while \ref{K_quatro} and \ref{K_cinco} are required only in Section \ref{s_proof}.\\
	
	\noindent \textbf{Notation:} In this paper, we use the following notations:
	\begin{itemize}
		\item The usual norm in $L^{p}(\mathbb{R}^N)$ is denoted by $\|\cdot  \| _p;$
		\item $B_R(x_0)$ is the $N$-ball of radius $R$ and center $x_0;$ $B_R:=B_R(0);$
		\item  $C_i$ denotes (possibly different) any positive constant;
		\item $\mathcal{X}_A$ is the characteristic function of the set $A \subset \mathbb{R}^N;$
		\item $A^c =\mathbb{R}^N \setminus A,$ for $A \subset \mathbb{R}^N;$ 
		\item $|A|$ is the Lebesgue measure of the measurable set $A \subset \mathbb{R}^N;$
		\item $a_n = b_n + o_n (1)$ if and only if $\lim _{n \rightarrow \infty }(a_n - b_n) = 0;$
		\item For any measurable function $\phi : \Omega \rightarrow \mathbb{R}$, we define $[\phi > \lambda] := \{ x \in \Omega : \phi(x) > \lambda \}.$ Analogous notation is adopted for other types of inequalities;
		\item $\phi ^+(x) = \max\{0,\phi (x) \}$ and $\phi ^{-} (x) = \max\{ 0, -\phi (x)  \}.$
	\end{itemize}
	\section{Preliminaries}\label{s_preli}
For $s \in (0,1)$, we denote by $D^{s,2}(\mathbb{R}^{N})$ the  homogeneous fractional Sobolev space, which is defined as 
\begin{equation*}
	D^{s,2}(\mathbb{R}^{N}):=\left\{u \in L^{2^{\ast}_s}(\mathbb{R}^{N}): \int_{\mathbb{R}^{N}}\int_{\mathbb{R}^{N}}\frac{(u(x)-u(y))^{2}}{|x-y|^{N+2s}}\dxdy<+ \infty \right\}.
\end{equation*}
It is known that $D^{s,2}(\mathbb{R}^{N})$ is a Hilbert space with inner product
	\begin{equation*}
		[u,v]_{D^{s,2}(\mathbb{R}^{N})}=\int_{\mathbb{R}^{N}}\int_{\mathbb{R}^{N}}\frac{(u(x)-u(y))(v(x)-v(y))}{|x-y|^{N+2s}}\dxdy,
	\end{equation*}
and norm $\|u\|_{D^{s,2}(\mathbb{R}^N)} = \sqrt{[u,u]_{D^{s,2}(\mathbb{R}^{N})} }.$ We also take into account $D_{K_s}^{s,2}(\mathbb{R}^{N})$ as the subspace of $L^{2^\ast _s} (\mathbb{R}^N)$ defined as the space of measurable functions $u$ such that the map $(x,y)\mapsto(u(x)-u(y))\sqrt{K_s(x-y)}$ belongs to $L^{2}(\mathbb{R}^{N}\times \mathbb{R}^{N})$. It is known (cf. \cite{ambrosio2016,duarte1,iane1}) that $D_{K_s}^{s,2}(\mathbb{R}^{N})$ is characterized as the completion of $C_0^\infty  (\mathbb{R}^N) $ with respect to the norm
\begin{equation*}
	[ u ]:=\| u \|_{D_{K_s}^{s,2}(\mathbb{R}^{N})} :=\left(\int_{\mathbb{R}^{N}}\int_{\mathbb{R}^{N}}(u(x)-u(y))^{2}K_s(x-y)\dxdy\right)^{\frac{1}{2}}.	
\end{equation*}
Moreover, $(D_{K_s}^{s,2}(\mathbb{R}^{N}), \|\cdot\|_{D_{K_s}^{s,2}(\mathbb{R}^{N})})$ is a Hilbert space equipped with the inner product
\begin{equation*}
	[u,v]:=[u,v]_{D_{K_s}^{s,2}(\mathbb{R}^{N})}:=\int_{\mathbb{R}^{N}}\int_{\mathbb{R}^{N}}(u(x)-u(y))(v(x)-v(y))K_s(x-y) \dxdy.
\end{equation*}
The space $D_{K_s}^{s,2}(\mathbb{R}^{N})$ embeds continuously into $D^{s,2}(\mathbb{R}^{N})$ (cf. \ref{K_dois}) and, consequently, into $L^{2^{\ast}_s}(\mathbb{R}^{N})$. Due to the vanishing nature of the potential $V$ at infinity, the natural energy space for our problem is defined as
\begin{equation*}
E=\left\{u \in D_{K_s}^{s,2}(\mathbb{R}^{N}): \int_{\mathbb{R}^{N}}V(x)u^{2}dx< + \infty \right\}.
\end{equation*}
The space $E$ is a Hilbert space, which is continuously embedded in $L^{2^\ast _s} (\mathbb{R}^N),$ when equipped with the inner product
\begin{equation*}
	(u,v) = \int_{\mathbb{R}^{N}}\int_{\mathbb{R}^{N}}(u(x)-u(y))(v(x)-v(y))K_s(x-y) \dxdy + \int_{\mathbb{R}^N}V(x) uv \dx,
\end{equation*}
and the associated norm $\| u \| = \sqrt{(u,u)}.$ On the other hand, for any measurable sets $A, B \subseteq \mathbb{R}^{N}$ and functions $u,v \in E$, we adopt the following notation for the nonlocal term:
\begin{equation*}
	[u,v]_{A\times B}=\int_{A}\int_{B}(u(x)-u(y))(v(x)-v(y))K_s(x-y)dxdy.
\end{equation*}
For the sake of conciseness, we denote $[u,v]_{\mathbb{R}^{N} \times \mathbb{R}^{N}}$ simply by $[u,v]$.
	\begin{proposition}\label{p_ge} We have the following consequences of hypotheses \ref{f_um}--\ref{f_dois}:
		\begin{enumerate}[label=\bf \roman*):]
			\item There are constants $A_1,$ $A_2>0$ such that $f(t) \leq A_1 t^{p-1}$ and $f(t) \leq A_2 t ^{2^\ast _s -1}.$
			\item There exist $C_1, C_2 > 0$ such that $F(t)\geq C_1t^{\theta}-C_2.$
		\end{enumerate}
	\end{proposition}
Weak solutions to Eq. \eqref{P} are defined as functions $u \in E$ satisfying
\begin{equation*}
	(u,v) = \int_{\mathbb{R}^{N}} (f(u) + \lambda |u|^{q-2}u)v \dx, \quad \forall v \in E.
\end{equation*}
Formally, the Euler-Lagrange functional associated with \eqref{P} is given by
\begin{equation*}
	I(u) = \frac{1}{2}\|u\|^{2} - \int_{\mathbb{R}^{N}} F(u) \dx - \frac{\lambda}{q} \int_{\mathbb{R}^N} |u|^{q} \dx.
\end{equation*}
However, for $\lambda \neq 0$, this functional is not well-defined on the entire space $E$ because the exponent $q$ is supercritical, i.e., $q > 2^\ast _s$. To overcome the lack of integrability, we consider a modified problem alongside the auxiliary functional $I_0: E \rightarrow \mathbb{R}$. This functional, which serves as a reference for our energy estimates, is defined by
\begin{equation*}
	I_0(u) = \frac{1}{2}\|u\|^{2} - \int_{\mathbb{R}^{N}} F(u) \dx.
\end{equation*}
\begin{proposition}
		$I_0 \in C^1 (E,\mathbb{R})$ and has the mountain pass geometry, more precisely
		\begin{enumerate}[label=\bf \roman*):]
			\item There exist $r_0,$ $b_0>0$ such that $I_0(u) \geq b_0,$ whenever $\|u\| = r_0;$
			\item There is $e_0 \in E$ with $\| e_0 \| > r_0$ and $I_0(e_0) < 0.$
		\end{enumerate}
		In particular, one can consider the related minimax level given by
		\begin{equation*}
			c(I_0)=  \inf_{\gamma \in \mathcal{M} ( I_0) } \max_{t \in [0,1]}I_0(\gamma(t)),
		\end{equation*}
		where $\mathcal{M} ( I_0)=\left\{\gamma \in C([0,1],E): \gamma(0)=0 \mbox{ and } I_0(\gamma(1))<0\right\}.$
	\end{proposition}
	\begin{proof}
We use Proposition \ref{p_ge}, combined with the embedding $E \hookrightarrow L^{2^*_s}(\mathbb{R}^N)$, to establish the required properties.

\textit{i)}: Let $u \in E.$ Then
\begin{equation*}
	I_0(u) \geq \frac{1}{2}\| u \| ^2 -\frac{A_2}{2^\ast _s} \int _{\mathbb{R}^N} |u|^{2_s^\ast }\dx \geq \| u \|^2 \left(\frac{1}{2} - C \| u \|^{2^\ast _s - 2} \right) >0,
\end{equation*}
for a suitable $C>0$ and $\| u \|>0$ small enough.

\textit{ii)}: Consider $\varphi \in C^\infty _0(\mathbb{R}^N)\setminus \{ 0 \},$ with $\varphi \geq 0,$ and denote $S = \supp(\varphi).$ Since $\theta \in (2,p),$ we have
\begin{equation*}
	I_0(t \varphi ) \leq \frac{t^2}{2} \| \varphi \|^2 - C_1 t^{\theta } \int_{S} \varphi ^{\theta } \dx + C_2 | S | \rightarrow - \infty,\text{ as }t \rightarrow \infty. \qedhere
\end{equation*}
	\end{proof}
	\section{Study of the auxiliary problem}\label{s_aux}
For each $k \in \mathbb{N}$ and $\lambda > 0$, we introduce the continuous function $f_{\lambda, k} : \mathbb{R} \rightarrow \mathbb{R}$ defined by
\begin{equation*}
	f_{\lambda, k}(t)=
	\left\{
	\begin{aligned}
	&0,  &&\text{ if } t<0, \\
	&f(t)+\lambda t^{q-1}, &&\text{ if } 0\leq t\leq k, \\
	&f(t)+\lambda k^{q-p}t^{p-1},&&\text{ if } t\geq k.
	\end{aligned}	
	\right.
\end{equation*}
Now, let $\nu = 2\theta / (\theta - 2)$, where $\theta > 2$ is given by \ref{f_dois}. To construct an appropriate auxiliary problem that recovers compactness at infinity, we first define the following localized nonlinearity:
\begin{equation*}
	\bar{f}_{\lambda, k}(x,t)=
	\left\{ 
	\begin{aligned}
		&f_{\lambda, k}(t), &&\text{ if} \quad  \nu f_{\lambda, k}(t) \leq V(x)t\text{ and }  t\geq 0, \\
		& \frac{1}{\nu} V(x)t ,&& \text{ if} \quad \nu f_{\lambda, k}(t)>V(x)t\text{ and }  t \geq 0, \\
		&0, && \text{ if} \quad x \in \mathbb{R}^N \text{ and }t < 0.
	\end{aligned}
	\right.
\end{equation*}
along with the Carath\'{e}odory function $g_{\lambda, k} : \mathbb{R}^N \times \mathbb{R} \rightarrow \mathbb{R},$
\begin{equation*}
	g_{\lambda, k}(x,t) =
	\left\{ 
	\begin{aligned}
		&f_{\lambda, k}(t), && \text{ if }  |x|\leq R, \\
		&\bar{f}_{\lambda, k}(x,t), && \text{ if }  |x|>R.
	\end{aligned}
	\right.
\end{equation*}
Consequently, we are led to consider the modified auxiliary problem
\begin{equation*}\tag{$P_{\lambda,k}$}\label{Plambda}
	\left\{  
	\begin{aligned}
		&\mathcal{L}_{K_s}u+V(x)u=g_{\lambda, k}(x,u) \text{ in }\mathbb{R}^{N}, \\
		&u \in E.
	\end{aligned}
	\right.
\end{equation*}
Let $F_{\lambda, k}(t) = \int_{0}^t f_{\lambda, k}(\tau)\dtau$ and $G_{\lambda, k}(x,t)=\int_{0}^{t}g_{\lambda, k}(x,\tau)\dtau.$ The auxiliary functions satisfy the following properties:
\begin{remark}\label{r1}
	\begin{enumerate}[label=\bf \roman*):]
		\item In view of \ref{f_um}, we can find $C_1, C_2 > 0$ satisfying
		\begin{equation*}
			f_{\lambda, k}(t)\leq C_{1}(1+\lambda k^{q-p})t^{p-1}\quad \text{and}\quad  f_{\lambda, k}(t)\leq C_{2}(1+\lambda k^{q-p})t^{2^{\ast}_s-1}.
		\end{equation*}
		\item We have
		\begin{equation*}
			tf_{\lambda, k}(t)-\theta F_{\lambda, k}(t) = 
			\left\{ 
			\begin{aligned}
				&0,&& \mbox{ if } t\leq 0, \\
				&f(t)t-\theta F(t)+\lambda\left(\frac{q-\theta}{q}\right)t^{q},&&  \mbox{ if }0\leq t\leq k, \\
				& f(t)t-\theta F(t)+\lambda k^{q-p}\frac{t^p}{p}(p-\theta)+\theta \lambda k^q\left(\frac{q-p}{qp}\right), && \text{if } t\geq k.
			\end{aligned}
			\right.
		\end{equation*}
		In particular, hypothesis \ref{f_dois} implies $tf_{\lambda, k}(t)-\theta F_{\lambda, k}(t) \geq 0$. Moreover, $F_{\lambda,k} (t) \geq F(t).$
		\item If $|x|\leq R$, then $g_{\lambda, k}(x,t)=f_{\lambda,k}(t)$ and $G_{\lambda, k}(x,t)=F_{\lambda,k}(t).$ 
		\item If $|x| > R$ and $t \geq 0,$ then $\bar{f}_{\lambda, k}(x,t) \leq f_{\lambda, k}(t).$ Furthermore,
		\begin{equation*}
			g_{\lambda, k}(x,t) \leq \frac{V(x)}{\nu}t, \quad g_{\lambda, k}(x,t)t \leq \frac{V(x)}{\nu}t^2, \quad \text{and} \quad G_{\lambda, k}(x,t)\leq \frac{V(x)}{2 \nu}t^{2}.
		\end{equation*}
		In particular, we have $G_{\lambda, k}(x,t) \leq F_{\lambda,k}(t)$ for all $|x| > R.$
		\end{enumerate}
\end{remark}
The energy functional $J_{\lambda, k}: E \rightarrow \mathbb{R}$ associated with Eq. \eqref{Plambda} is defined as
\begin{equation*}
	J_{\lambda, k}(u) = \frac{1}{2}\|u\|^{2} - \int_{\mathbb{R}^{N}} G_{\lambda, k}(x,u)\dx.
\end{equation*}
Standard arguments involving Remark \ref{r1} ensure that $J_{\lambda, k}$ is well defined with $J_{\lambda, k} \in C^{1}(E, \mathbb{R})$, and its Fréchet derivative is given by
\begin{equation}\label{def_ws}
	J_{\lambda, k}'(u) \cdot  v  = (u,v) - \int_{\mathbb{R}^{N}} g_{\lambda, k}(x,u)v\dx, \quad \forall \, u, v \in E.
\end{equation}
Solutions of \eqref{Plambda} are defined as critical points of $J_{\lambda, k}.$
\begin{proposition}
	\begin{enumerate}[label=\bf \roman*):]
		\item There exist $r_{\lambda, k},$ $b_{\lambda, k}>0$ such that $J_{\lambda, k}(u) \geq b_{\lambda, k},$ whenever $\|u\| = r_{\lambda, k};$
		\item There is $e_{\lambda, k} \in E$ with $\| e_{\lambda, k} \| > r_{\lambda, k}$ and $J_{\lambda, k}(e_{\lambda, k}) < 0.$
	\end{enumerate}
The related minimax level is defined by
	\begin{equation*}
		c(J_{\lambda, k} )= \inf_{  \gamma \in \mathcal{M}( J_{\lambda, k }) }\max_{ t \in [0,1]  } J_{\lambda, k}(\gamma(t)),
	\end{equation*}
where $\mathcal{M}( J_{\lambda, k }) =\left\{\gamma \in C([0,1],E): \gamma(0)=0 \mbox{ and } J_{\lambda, k}(\gamma(1))<0\right\}.$
\end{proposition}
\begin{proof}
	Clearly, Remark \ref{r1} implies the following inequality $G_{\lambda, k } (x,t) \leq  (C_2/2^\ast _s) (1+\lambda k ^{q-p}) t^{2^\ast _s}.$ This inequality together with the embedding $E \hookrightarrow L^{2^\ast _s} (\mathbb{R}^N)$ leads to
	\begin{equation*}
		J_{\lambda , k } (u) \geq  \| u \|^2 \left(   \frac{1}{2}- C_{\lambda,k} \| u \|^{2^\ast _s -2} \right) >0,
	\end{equation*}
	for $\| u \| $ sufficiently small. Conversely, let $\varphi \in C^\infty _0(\mathbb{R}^N)\setminus \{ 0 \}$ be such that $\varphi \geq 0$ and set $S  = \supp(\varphi)\subset B_R.$ Recalling that $F_{\lambda, k} (t) \geq F(t)$ and using Proposition \ref{p_ge}, we obtain
	\begin{equation*}
		J_{\lambda, k} (t \varphi ) \leq \frac{t^2}{2} \| \varphi \|^2 - C_1 t^{\theta } \int _{S} \varphi ^{\theta } \dx + C_2 |S|\rightarrow - \infty,\text{ as }t \rightarrow \infty. \qedhere
	\end{equation*}
\end{proof}
Consequently, the Mountain Pass Theorem (without the Palais–Smale condition) provides the existence of a sequence $(u_{n}) \subset E$ such that
\begin{equation}\label{ps}
	J_{\lambda, k}(u_{n}) \rightarrow c_{\lambda, k} \quad \text{and} \quad J_{\lambda, k}'(u_{n}) \rightarrow 0 \text{ in } E^{\ast}.
\end{equation}
	\begin{lemma}\label{l_psbounded}
		The sequence $(u_{n })$ is bounded.
	\end{lemma}
	\begin{proof}
		By Remark \ref{r1} and the fact that $\nu = 2\theta / (\theta - 2)$, we have
		\begin{align*}
			J_{\lambda,k}(u)-\frac{1}{\theta}J_{\lambda,k}'(u)\cdot u & = \left(\frac{\theta-2}{2\theta}\right)\|u\|^{2}+ \int_{\mathbb{R}^{N}}\frac{1}{\theta}g_{\lambda, k}(x,u)u-G_{\lambda, k}(x,u)\dx \\
			&\geq \left(\frac{\theta-2}{2\theta}\right)\|u\|^{2}+ \int_{B^c_R}\frac{1}{\theta}g_{\lambda,k}(x,u)u - \frac{1}{2\nu}\int_{B^c_R}V(x)u^{2}\dx \geq \left(\frac{\theta-2}{2\theta}\right)\|u\|^{2}. 
		\end{align*}
		Consequently,
		\begin{equation*}
			|J_{\lambda,k}(u)|+\|J_{\lambda,k}'(u)\|\| u \|\geq \left(\frac{\theta-2}{2\theta}\right)\|u\|^{2},
		\end{equation*}
		for all $u \in E$. This inequality ensures that the sequence is bounded.
	\end{proof}
Remark \ref{r1} also implies $J_{\lambda,k}(u)\leq I_0(u),$ $u \in E.$ In particular, we have $c(J_{\lambda, k }) \leq c(I_0)$ and the inequalities of Lemma \ref{l_psbounded} lead to:
	\begin{align}
c(I_0)&\geq c(J_{\lambda, k }) =J_{\lambda,k}( u_n )-\frac{1}{\theta}J_{\lambda,k}'(u_n  ) \cdot u_n + o_n(1) \nonumber \\
&\geq \left(\frac{\theta-2}{2\theta}\right)\| u_n \| ^{2} + o_n(1)\nonumber\\
&\geq  \left(\frac{\theta-2}{2\theta}\right) \mathbb{S}_{K_s} \| u_n  \|_{2^{\ast}_s}^{2} + o_n(1), \label{est1}
	\end{align}
where $\mathbb{S}_{K_s}$ is the Sobolev constant (see \ref{K_tres}) given by
\begin{equation}\label{constS}
	\mathbb{S}_{K_s} = \inf \left\{ \| u \|^2_{D_K^{s,2}(\mathbb{R}^{N})}: u \in D_K^{s,2}(\mathbb{R}^{N})  \text{ and }\| u \|_{2_s^\ast} = 1\right\}.
\end{equation}
We now proceed to prove that $J_{\lambda,k}$ satisfies the Palais-Smale condition at the mountain pass level, or more precisely, that $(u_n)$ has a convergent subsequence. To this end, we establish some technical results necessary to control the nonlocal behavior of the operator $\mathcal{L}_{K_s}$. 

Let us introduce a smooth cutoff function $\eta \in C^\infty(\mathbb{R}^N, [0,1])$ satisfying $\eta=1$ in $B_{2r}^{c}$, $\eta=0$ in $B_{r}$, and $|\nabla \eta| \leq 2/r.$ We also introduce the set $A_r := \{x \in \mathbb{R}^{N}: r \leq  |x| < 2r\},$ for $r > R$. The proof of the subsequent result follows the arguments presented in \cite{duarte1} (see Lemmas 3.1--3.6 therein), and thus we omit the details here.
\begin{lemma}\label{lem1}
		\begin{enumerate}[label=\bf \roman*):]
		\item  $\displaystyle [u_n ,  \eta u_n] _{B_r \times B^c_{2r}}   + [u_n, \eta u_n ]_{B^c_{2r}  \times B_r } \geq-\int_{B_{r}}\int_{B_{2r}^{c}}u_{n}(y)^{2}K_s(x-y)\dxdy.$
		\item For a given $\varepsilon>0,$ there is $r_{0}>0$ such that, if $r>r_{0},$ then
		\begin{equation*}
			\int_{B_{r}}\int_{B_{2r}^{c}}u_{n}(y)^{2}K_s(x-y)\dxdy<\varepsilon.
		\end{equation*}
		\item There are $C'>0$ and $C''>0$ such that
		\begin{align*}
			\int_{A_r}\int_{\mathbb{R}^{N}}|u_{n}(y)|   &|(u_{n}(x)-u_{n}(y))|   |(\eta(x)-\eta(y))|    K_s(x-y)\dxdy \\
			&\leq \frac{C'}{r}\|u_{n}\|_{L^{2}(A_r)}[u_n] + C''\|u_{n}\|_{L^{2}(A_r)}[u_n].
		\end{align*}
		\item For the same constants $C'$ and $C''$ above, we have
		\begin{align*}
			\int_{B_{r}}\int_{A_r}|u_{n}(x)-u_{n}(y)|& |\eta(x)u_{n}(x)-\eta(y)u_{n}(y)|K_s(x-y)\dxdy \\
			&\leq \frac{C'}{r}\|u_{n}\|_{L^{2}(A_r)}[u_{n}] + C''\|u_{n}\|_{L^{2}(A_r)}[u_{n}].
		\end{align*}
		\item Similarly,
		\begin{align*}
	-\int_{B_{2r}^{c}}\int_{A_r}u_{n}(y)&(u_{n}(x)-u_{n}(y))(\eta(x)-\eta(y))K_s(x-y)\dxdy \\
	&\leq \frac{C'}{r}\|u_{n}\|_{L^{2}(A_r)}[u_{n}] + C''\|u_{n}\|_{L^{2}(A_r)}[u_{n}].
		\end{align*}
		
	\end{enumerate}
\end{lemma}

	\begin{lemma}\label{prop1}
		There is $C>0$, such that
		\begin{equation*}
			[u_{n},\eta u_{n}] \geq -\int_{B_{r}}\int_{B_{2r}^{c}}u_{n}(y)^{2}K_s(x-y)\dxdy -C((1/r)+1)\|u_{n}\|_{L^{2}(A_r)}.
		\end{equation*}
	\end{lemma}
	\begin{proof} Since $\eta=0$ in $B_{r}$, we have 
		\begin{align*}
			[u_{n},\eta u_{n}]=	&[u_{n},\eta u_{n}]_{B_{r}\times A_r} +[u_{n},\eta u_{n}]_{B_{r}\times B_{2r}^{c}}+ [u_{n},\eta u_{n}]_{A_r \times \mathbb{R}^{N}}\\
			&+[u_{n},\eta u_{n}]_{B_{2r}^{c}\times B_{r}}+[u_{n},\eta u_{n}]_{B_{2r}^{c}\times A_r} + [u_{n},\eta u_{n}]_{B_{2r}^{c}\times B_{2r}^{c}} .
		\end{align*}
		Since $\eta = 1$ on $B_{2r}^{c}$, it follows that $[u_{n},\eta u_{n}]_{B_{2r}^{c}\times B_{2r}^{c}} \geq 0$. Moreover, invoking Lemma \ref{lem1}--i), we obtain
		\begin{align*}
			[u_{n},\eta u_{n}]_{A_r \times \mathbb{R}^{N}}&+[u_{n},\eta u_{n}]_{B_{2r}^{c}\times A_r}\\
			&\leq [u_{n},\eta u_{n}]+ \int_{B_{r}}\int_{B_{2r}^{c}}u_{n}(y)^{2}K_s(x-y)\dxdy-[u_{n},\eta u_{n}]_{B_{r}\times A_r}.
		\end{align*}
		On the other hand, for $C$ and $D$ subsets of $\mathbb{R}^{N}$, the following identity holds,
		\begin{align*}
				[u, \eta u]_{C \times D} &=\int_{C}\int_{D}(u(x)-u(y))(\eta(x) u(x)-\eta (y) u(y))K_s(x-y)\dxdy \\
			&=\int_{C} \int_{D}\eta(x)(u(x)-u(y))^{2}K_s(x-y)\dxdy\\
			&\qquad  + \int_{C}\int_{D}u(y)(u(x)-u(y))(\eta(x)-\eta(y))K_s(x-y)\dxdy,
		\end{align*}
for any $u \in E.$ Consequently,
\begin{align*}
	\int_{A_r} &\int_{\mathbb{R}^{N}}\eta(x)(u_{n}(x)-u_{n}(y))^{2}K_s(x-y)\dxdy\\
	+&\int_{B_{2r}^{c}}\int_{A_r}\eta(x)(u_{n}(x)-u_{n}(y))^{2}K_s(x-y)\dxdy \\
	 &\leq [u_{n},\eta u_{n}] + \int_{B_{r}}\int_{B_{2r}^{c}}u_{n}(y)^{2}K_s(x-y)\dxdy-[u_{n},\eta u_{n}]_{B_{r}\times A_r} \\
	&\qquad -\int_{A_r}\int_{\mathbb{R}^{N}}u_{n}(y)(u_{n}(x)-u_{n}(y))(\eta(x)-\eta(y))K_s(x-y)\dxdy \\
	&\qquad   -\int_{B_{2r}^{c}}\int_{A_r}u_n(y)(u_{n}(x)-u_{n}(y))(\eta(x)-\eta(y))K_s(x-y)\dxdy.
\end{align*}
Now we apply items iii), iv) and v) of Lemma \ref{lem1}, to obtain some constants $C,$ $C_{1},$ $C_{2}>0$ such that
\begin{align*}
	&0\leq \int_{A_r}\int_{\mathbb{R}^{N}}\eta(x)(u_{n}(x)-u_{n}(y))^{2}K_s(x-y)\dxdy\\
	&\qquad+\int_{B_{2r}^{c}}\int_{A_r}\eta(x)(u_{n}(x)-u_{n}(y))^{2}K_s(x-y)\dxdy  \\
	& \leq [u_{n},\eta u_{n}]+ \int_{B_{r}}\int_{B_{2r}^{c}}u_{n}(y)^{2}K_s(x-y)\dxdy \\
	&\qquad+\frac{C_{1}}{r}\|u_{n}\|_{L^{2}(A_r)}[u_{n}] + C_{2}\|u_{n}\|_{L^{2}(A_r)}[u_{n}]\\
	&\leq [u_{n},\eta u_{n}]+ \int_{B_{r}}\int_{B_{2r}^{c}}u_{n}(y)^{2}K_s(x-y)\dxdy +C((1/r)+1)\|u_{n}\|_{L^{2}(A_r)}.
\end{align*}
This proves the result.
\end{proof}
With the preceding technical results established, we are now ready to prove that $J_{\lambda,k}$ satisfies the Palais-Smale condition at the level $c(J_{\lambda,k})$. Since the sequence $(u_n)$ is bounded in $E$ by Lemma \ref{l_psbounded}, passing to a subsequence if necessary, we may assume that there exists $u_{\lambda,k} \in E$ such that $u_n \rightharpoonup u_{\lambda,k}$ weakly in $E$ and $u_n(x) \to u_{\lambda,k}(x)$ a.e. in $\mathbb{R}^N$. As a first step, we prove that this weak limit is a weak solution of \eqref{Plambda}.
\begin{lemma}
	$J'_{\lambda, k } (u_{\lambda, k })\cdot v = 0,$ for any $v \in E.$
\end{lemma}
\begin{proof}
	By Remark \ref{r1}, for a given $r>R,$ it is clear that $|g_{\lambda,k}(x,u_n)v| \leq (1/\nu) V(x)|u_n| |v|$ in $B^c_r$, with the same inequality holding for $u_{\lambda,k}$ instead of $u_n$. Applying the Cauchy-Schwarz inequality, and using the fact that $(u_n)$ is bounded in $E$, we obtain
	\begin{equation*}
		\int_{B_{r}^c} |g_{\lambda,k}(x,u_n)v| \dx \leq \frac{1}{\nu} \left( \int_{B_r^c} V(x)u_n^2 \dx \right)^{\frac{1}{2}} \left( \int_{B_r^c} V(x)v^2 \dx \right)^{\frac{1}{2}} \leq C \left( \int_{B_r^c} V(x)v^2 \dx \right)^{\frac{1}{2}},\quad \forall \, n \in \mathbb{N}.
	\end{equation*}
	Since $v \in E,$ the right-hand side vanishes as $r \to \infty$. Thus, for a given $\varepsilon>0,$ there exists $r_\varepsilon >R$ such that
	\begin{equation}\label{ha1}
		\left| \int _{B_{r_\varepsilon}^c} g_{\lambda, k }(x,u_n)v \dx  \right| <\frac{\varepsilon}{4} \quad\text{and}\quad  \left| \int _{B_{r_\varepsilon}^c} g_{\lambda, k }(x,u_{\lambda,k})v \dx  \right| <\frac{\varepsilon}{4},\quad \forall \, n \in \mathbb{N}.
	\end{equation}
	Next, using Remark \ref{r1} again, we know that $g_{\lambda, k }(x,t) \leq C_{\lambda, k } t^{p-1}.$ This allows us to apply a standard argument involving the Lebesgue Convergence Theorem in $B_{r_\varepsilon}$ to obtain the existence of $n_0=n_0(r_\varepsilon ,\varepsilon)$ such that
	\begin{equation}\label{ha2}
		\left| \int _{B_{r_\varepsilon}} g_{\lambda, k }(x,u_n)v \dx - \int _{B_{r_\varepsilon}} g_{\lambda, k }(x,u_{\lambda,k})v \dx \right| < \frac{\varepsilon}{2},\quad \text{whenever }n \geq n_0.
	\end{equation}
	Combining \eqref{ha1} and \eqref{ha2} yields
	\begin{equation*}
		\lim _{n \rightarrow \infty } \int _{\mathbb{R}^N }	g_{\lambda, k }(x,u_n)v \dx  = \int _{\mathbb{R}^N } g_{\lambda, k }(x,u_{\lambda,k})v  \dx .
	\end{equation*}
	Therefore, the conclusion follows by combining this limit with the weak convergence $u_n \rightharpoonup u_{\lambda,k}$ in $E$ and the fact that $\| J'_{\lambda, k } (u_n)\|_{E^\ast} \rightarrow 0$.
\end{proof}

\begin{proposition}\label{p_conv}
	There exists $u_{\lambda,k} \in E$ such that, up to a subsequence, $u_n \rightarrow u_{\lambda,k}$ in $E$.
\end{proposition}
\begin{proof}
To establish the strong convergence $u_n \rightarrow u_{\lambda,k}$ in $E$, since $E$ is a Hilbert space, it suffices to show that $\| u_n \| \rightarrow \| u_{\lambda,k} \|$. Given that $J'_{\lambda,k}(u_n) u_n = o_n(1)$, this is equivalent to proving that
	\begin{equation}\label{conv_fim}
		\lim\limits_{n \rightarrow \infty}\int_{\mathbb{R}^{N}}g_{\lambda,k}(x,u_{n})u_{n} \dx = \int_{\mathbb{R}^{N}}g_{\lambda,k}(x,u_{\lambda,k})u_{\lambda,k} \dx.
	\end{equation}
	To verify \eqref{conv_fim}, we estimate the behavior of the sequence outside a large ball. We observe that $\eta u_{n} \in E$ and satisfies $\|\eta u_{n}\| \leq C \|u_{n}\|$ (see Lemma 2.5 in \cite{duarte2}). Thus, $(\eta u_{n})$ is bounded in $E$, which implies $J_{\lambda,k}'(u_{n})\cdot (\eta u_{n})=o_{n}(1)$. Explicitly,
	\begin{equation}\label{eq_eta}
		[u_{n},\eta u_{n}]+\int_{\mathbb{R}^{N}}V(x)u_{n}^{2}\eta \dx = \int_{\mathbb{R}^{N}}g_{\lambda,k}(x,u_{n})(\eta u_{n})\dx + o_{n}(1).
	\end{equation}
	 Using the fact that $\eta = 1$ in $B_{2r}^c$ and $\eta = 0$ in $B_r$, we have $\int_{B_r^c} V(x) u_n^2 \eta \dx \geq \int_{B_{2r}^c} V(x) u_n^2 \dx .$ Simultaneously,  Remark \ref{r1} yields $g_{\lambda,k}(x,u_{n})u_{n} \leq (V(x) / \nu ) u_{n}^{2},$ for $|x|>r.$ Inserting these bounds and the lower estimate for $[u_n, \eta u_n]$ from Lemma \ref{prop1} into \eqref{eq_eta}, we deduce
	\begin{align*}
		\left(1-\frac{1}{\nu }\right)\int_{B^c_{2r}}V(x)u_{n}^{2} \dx & \leq \int_{B_{r}}\int_{B_{2r}^{c}}u_{n}(y)^{2}K_s(x-y)\dx\dy \nonumber \\
		&\quad + C_0\left( \frac{1}{r}+1 \right)\|u_{n}\|_{L^{2}(A_r)} +o_{n}(1),
	\end{align*}
	for some constant $C_0>0$ independent of $n$ and $r$. Let $\varepsilon>0$ be arbitrary. By Lemma \ref{lem1} \textit{ii)}, we can choose $r>R$ large enough such that the double integral is less than $(\nu -1)(\varepsilon /4)$. This yields
	\begin{equation}\label{eq:5}
		\int_{B^c_{2r}}g_{\lambda,k}(x,u_{n})u_{n}\dx \leq \frac{1}{\nu}\int_{B^c_{2r}}V(x)u_{n}^{2}\dx \leq \frac{\varepsilon}{4} + C'_0\|u_{n}\|_{L^{2}(A_r)} +o_{n}(1). 
	\end{equation}
	Furthermore, since $u_{\lambda,k} \in L^{2^\ast_s}(\mathbb{R}^N)$, Hölder's inequality over bounded domains implies that $\|u_{\lambda,k}\|_{L^2(A_r)} \leq C_N r^s \|u_{\lambda,k}\|_{L^{2^\ast_s}(A_r)} \to 0$ as $r \to \infty$. Thus, by enlarging $r$ if necessary, we can ensure that
	\begin{equation}\label{sob1}
		\|u_{\lambda,k}\|_{L^{2}(A_r)}< \frac{\varepsilon}{8 C'_0} \quad \text{and} \quad \int_{B_{2r}^c}g_{\lambda,k}(x,u_{\lambda,k})u_{\lambda,k}\dx < \frac{\varepsilon}{4},
	\end{equation}
	where we used the growth $g_{\lambda , k }(t) \leq C_{\lambda,k}t^{2^\ast_s -1}$ given by Remark \ref{r1} to ensure the finiteness of the integral. Fixing this $r$, the local compactness of the fractional Sobolev embedding implies that $u_n \rightarrow u_{\lambda,k}$ in $L^2(A_r)$. Hence, there exists $n_1 \in \mathbb{N}$ such that, for all $n > n_1$,
	\begin{equation}\label{eq:6}
		\| u_n \|_{L^2(A_r)} \leq \| u_n - u_{\lambda,k} \|_{L^2(A_r)} + \|u_{\lambda,k}\|_{L^2(A_r)} < \frac{\varepsilon}{8 C'_0} + \frac{\varepsilon}{8 C'_0} = \frac{\varepsilon}{4 C'_0}.
	\end{equation}
	Combining \eqref{eq:5} and \eqref{eq:6}, we obtain for $n > n_1$:
	\begin{equation}\label{fim1}
		\int_{B_{2r}^c}g_{\lambda,k}(x,u_{n})u_{n} \dx \leq \frac{\varepsilon}{2} + o_n(1). 
	\end{equation}
	Finally, on the bounded domain $B_{2r}$, the subcritical growth $g_{\lambda, k }(x,t) \leq C_{\lambda, k } t^{p-1}$ (Remark \ref{r1}) and the strong convergence $u_n \rightarrow u_{\lambda,k}$ in $L^p(B_{2r})$ guarantee the existence of $n_2 \in \mathbb{N}$ such that
	\begin{equation}\label{fim3}
	\left| \int_{B_{2r}}g_{\lambda,k}(x,u_{n})u_{n} \dx -   \int_{B_{2r}}g_{\lambda,k}(x,u_{\lambda,k})u_{\lambda,k} \dx \right| < \frac{\varepsilon}{4},
	\end{equation}
	whenever $n>n_2,$ up to a subsequence. Consequently, for all $n > \max\{n_1, n_2\}$, combining the estimates over $B_{2r}$ and $B_{2r}^c$ established in \eqref{sob1}, \eqref{fim1}, and \eqref{fim3} yields that the absolute difference of the integrals in \eqref{conv_fim} is strictly less than $\varepsilon$. Since $\varepsilon > 0$ is arbitrary, the convergence in \eqref{conv_fim} follows, which completes the proof.
\end{proof}
At this stage, it follows from \eqref{ps} and the strong convergence that $u_{\lambda,k} \in E$ solves \eqref{Plambda} with $J_{\lambda,k}(u_{\lambda,k})=c(J_{\lambda,k})$. Next, we apply a version of the Maximum Principle established in \cite{duarte2}. Although stated slightly differently therein, the proof can be readily adapted to our setting since it depends exclusively on the structural hypotheses \ref{K_um}--\ref{K_tres}.
\begin{lemmaletter}\cite[Proposition 4.1 and Theorem 4.7]{duarte2}\label{l_max}
	Assume \ref{K_um}--\ref{K_tres} and let 
	\begin{equation*}
	E_a =  \left\lbrace  u \in D_{K_s}^{s,2}(\mathbb{R}^{N}) :\int _{\mathbb{R}^N} a(x) u^2 \dx < + \infty   \right\rbrace,
	\end{equation*}
	where $a \in L^\infty _{\loca}(\mathbb{R}^N)$ with $a(x) \geq 0$ a.e. in $\mathbb{R}^N.$ Suppose that $u_0 \in E_a$ satisfies $ \mathcal{L}_{K_s} u_0 + a(x)u_0 \geq 0$ in $ \mathbb{R}^N,$ more precisely,
	\begin{equation}\label{max}
		\int_{\mathbb{R}^N}\int_{\mathbb{R}^N}(u_{0}(x)-u_{0}(y))(\phi(x)-\phi(y))K_s(x-y)\dxdy+\int_{\mathbb{R}^N}a(x)u_{0}\phi \dx \geq 0,
	\end{equation}
	for all $\phi \in E_a$ with $\phi \geq 0.$ Then either $u_0>0$ a.e. in $\mathbb{R}^N,$ or $u_0=0$ a.e. in $\mathbb{R}^N.$
\end{lemmaletter}
Since $g_{\lambda,k}(x,t) \geq 0$, the solution $u_{\lambda,k}$ also satisfies $\mathcal{L}_{K_s}u_{\lambda,k}+V(x)u_{\lambda,k}\geq 0$ in $\mathbb{R}^N,$ in the sense of \eqref{max} (see \eqref{def_ws}). Consequently, we conclude that $u_{\lambda,k}>0$ a.e. in $\mathbb{R}^N$.

\section{Moser iteration and $L^{\infty}$ estimates}\label{s_decay}
This section is devoted to establishing a uniform bound for the auxiliary solution $u_{\lambda,k}$ via the Moser iteration technique. We follow the approach developed in \cite{duarte1}, which relies on the properties of specific auxiliary functions. Let $\beta>1$ and define
\begin{equation*}
	\zeta_1(t) = t|t|^{2(\beta-1)} \quad \text{and} \quad \zeta_2(t) = t|t|^{\beta-1}.
\end{equation*}
Consider $t, \tau \in \mathbb{R}$ with $t \neq \tau$. By the Mean Value Theorem, there exist $\theta_{1}(t, \tau), \theta_{2}(t, \tau) \in \mathbb{R}$ such that
\begin{equation}\label{eq:12}
	\zeta_1'(\theta_{1}(t, \tau))=\frac{\zeta_1(t)-\zeta_1(\tau)}{t-\tau} \quad \text{and} \quad \zeta_2'(\theta_{2}(t, \tau))=\frac{\zeta_2(t)-\zeta_2(\tau)}{t-\tau}.
\end{equation}
Solving for $|\theta_{i}(t, \tau)|$ ($i=1,2$) in the expressions above, we obtain the explicit formulas:
\begin{equation*}
	|\theta_{1}(t, \tau)|=\left(\frac{1}{2\beta-1}\frac{t|t|^{2(\beta-1)}-\tau | \tau |^{2(\beta-1)}}{t- \tau }\right)^{\frac{1}{2(\beta-1)}},
\end{equation*}
and
\begin{equation*}
	|\theta_{2}(t, \tau)|=\left(\frac{1}{\beta}\frac{t|t|^{\beta-1}-\tau | \tau |^{\beta-1}}{t- \tau }\right)^{\frac{1}{\beta-1}}.
\end{equation*}
The following lemma establishes a fundamental inequality relating these functions.
\begin{lemmaletter}[\cite{duarte1}, Lemmas 4.3 and 4.4]\label{lem42}
$|\theta_{1}(t, \tau)|\geq|\theta_{2}(t, \tau)|.$
\end{lemmaletter}

Our argument also makes use of a technical inequality related to the truncation levels. Let
\begin{equation*}
	h(t,\tau ):=2(nt-\tau | \tau |^{\beta-1})^{2}-C_\beta (t-\tau )(n^{2}t-\tau | \tau |^{2(\beta-1)}),
\end{equation*}
with $n \in \mathbb{N}$ and $C_\beta= 2 + \frac{2(\beta-1)^{2}}{2\beta-1}.$ We prove the following result in Appendix \ref{s_app_inequality}.

\begin{lemma}\label{rem1}
	We have $h(t, \tau )\leq 0$, provided that $|t|>n^{1/(\beta -1)} $ and $|\tau |\leq n^{1/(\beta -1)}.$
\end{lemma}
\begin{proposition}
For each $\lambda>0$ and $k \in \mathbb{N}$, the solution $u_{\lambda,k}$ of the auxiliary problem satisfies
\begin{equation*}
	\|u_{\lambda,k}\|_{\infty}\leq C\left(1+\lambda k^{q-p}\right)^{\frac{1}{2^{\ast}_s-p}}\|u_{\lambda,k}\|_{2^{\ast}_s},
\end{equation*}
where $C$ is a positive constant independent of $\lambda$ and $k$.
	\end{proposition}
	\begin{proof}
		For each $n \in \mathbb{N},$ define $A_{n}:=\left\{x \in \mathbb{R}^{N}; |u_{\lambda,k}(x)|^{\beta-1}\leq n\right\},$ $B_{n}:=A_n^c,$ 
		\begin{equation*}
			\zeta_{1,n} (t)=
			\left\{ 
			\begin{aligned}
			&t|t|^{2(\beta-1)},&\text{if }|t|^{\beta-1}\leq n,\\
			&n^{2}t,&\text{if }|t|^{\beta-1}>n,
			\end{aligned}
			\right.
			\quad \text{and} \quad 	
			\zeta_{2,n} (t)=
			\left\{ 
			\begin{aligned}
				&t|t|^{(\beta-1)},&\text{if }|t|^{\beta-1}\leq n,\\
				&nt,&\text{if }|t|^{\beta-1}>n.
			\end{aligned}
			\right.
		\end{equation*}
		Let	$v_{n}:= v_{n,\lambda, k}=\zeta_{1,n}\circ u_{\lambda,k}$ and $w_{n}:=w_{n,\lambda, k}=\zeta_{2,n}\circ u_{\lambda,k}$.	For each $x,y \in A_{n}$, choose $\theta_{1}(u_{\lambda,k}(x),u_{\lambda,k}(y))$ and $ \theta_{2}(u_{\lambda,k}(x),u_{\lambda,k}(y))$, satisfying \eqref{eq:12}. We denote $\theta_{i}(u_{\lambda,k}(x),u_{\lambda,k}(y))$ by $\theta_{i}(x,y),$ $i=1,2$. Then
		\begin{multline}\label{eq_der}
			v_{n}(x)-v_{n}(y)=(2\beta-1)|\theta_{1}(x,y)|^{2(\beta-1)}(u_{\lambda,k}(x)-u_{\lambda,k}(y))\text{ in }A_n\\ \text{ and }w_{n}(x)-w_{n}(y)=\beta|\theta_{2}(x,y)|^{\beta-1}(u_{\lambda,k}(x)-u_{\lambda,k}(y))\text{ in }A_n.
		\end{multline}
		Since $\zeta_{1,n}(0)=\zeta_{2,n}(0)=0$ and $\zeta'_{1,n}, \zeta'_{2,n} \in L^\infty (\mathbb{R}),$ a simple estimate involving \eqref{eq_der} shows that $v_n$ and $w_n$ belong to $E.$ In view of $\left[w_{n},w_{n}\right]_{A_{n}\times B_{n}} = \left[w_{n},w_{n}\right]_{B_{n}\times A_{n}}$ and $\left[u_{\lambda,k},v_{n}\right]_{A_{n}\times B_{n} } = \left[u_{\lambda,k},v_{n}\right]_{B_{n}\times A_{n} }$ (see \ref{K_um}), Lemma \ref{lem42} and the last identities imply
		\begin{multline*}
			\left[w_{n},w_{n}\right] -\left[u_{\lambda,k},v_{n}\right] \leq  (\beta-1)^{2}\int_{A_{n}}\int_{A_{n}}|\theta_{1}(x,y)|^{2(\beta-1)}(u_{\lambda,k}(x)-u_{\lambda,k}(y))^{2}K_s(x-y)\dxdy\\ 
			+2\left[w_{n},w_{n}\right]_{A_{n}\times B_{n}}-2\left[u_{\lambda,k},v_{n}\right]_{A_{n}\times B_{n} }.
		\end{multline*}
		Next, we use the facts that $\left[u_{\lambda,k},v_{n}\right] = \left[u_{\lambda,k},v_{n}\right]_{A_n \times A_n} +2\left[u_{\lambda,k},v_{n}\right] _{A_n \times B_n}+ \left[u_{\lambda,k},v_{n}\right] _{B_n \times B_n}$ (see \ref{K_um}) and $\left[u_{\lambda,k},v_{n}\right] _{B_n \times B_n} \geq 0$ to further estimate
		\begin{multline*}
			(\beta-1)^{2}\int_{A_{n}}\int_{A_{n}}|\theta_{1}(x,y)|^{2(\beta-1)}(u_{\lambda,k}(x)-u_{\lambda,k}(y))^{2}K_s(x-y)\dxdy\\
			+2\left[w_{n},w_{n}\right]_{A_{n}\times B_{n}}-2\left[u_{\lambda,k},v_{n}\right]_{A_{n}\times B_{n} }\\
			 \leq \frac{(\beta-1)^{2}}{2\beta-1}[u_{\lambda,k},v_{n}] +2\left[w_{n},w_{n}\right]_{B_{n}\times A_{n}}-\left(2+\frac{2(\beta-1)^{2}}{2\beta-1}  \right)\left[u_{\lambda,k},v_{n}\right]_{B_{n} \times A_{n}}.
			\end{multline*}
Taking  $x\in B_{n}$, $y \in A_{n}$, $t=u_{\lambda,k}(x)$ and $\tau =u_{\lambda,k}(y)$ in Lemma \ref{rem1} we obtain
\begin{equation*}
2\left[w_{n},w_{n}\right]_{B_{n}\times A_{n}}-\left(2+\frac{2(\beta-1)^{2}}{2\beta-1}  \right)\left[u_{\lambda,k},v_{n}\right]_{B_{n}\times A_{n}} \leq 0.	
\end{equation*}
Consequently, since $\beta > 1$, we have
\begin{equation*}
	\left[w_{n},w_{n}\right] \leq  \frac{\beta^2}{2\beta-1}[u_{\lambda,k}, v_n] \leq  \beta [u_{\lambda,k},v_{n}] .
\end{equation*}
		On the other hand, choosing $v = v_n$ in \eqref{def_ws} and recalling that $u_{\lambda,k}v_n=w_{n}^2$ yields
		\begin{equation*}
		\left[w_{n},w_{n}\right] +\int_{\mathbb{R}^N}V(x)w_{n}^2dx\leq \beta \int_{\mathbb{R}^N}g_{\lambda,k}(x,u_{\lambda,k})v_n \dx.
		\end{equation*}
		Now we use Remark \ref{r1}, \eqref{est1}, \eqref{constS} and $u_{\lambda,k}v_n=w_n^2$ to get the following estimate
		\begin{align*}
				\left(\int_{A_{n}}|w_{n}|^{2^{\ast}_{s}}\dx\right)^{\frac{2}{2^{\ast}_{s}}} & \leq \mathbb{S}_{K_s}\beta\int_{\mathbb{R}^N}g_{\lambda,k}(x,u_{\lambda,k})v_n  \dx \\
				&\leq \mathbb{S}_{K_s}\beta C_1(1+\lambda k^{q-p})\int_{\mathbb{R}^{N}}u_{\lambda,k}^{p-2}w_{n}^{2} \dx\\
				&\leq \mathbb{S}_{K_s}\beta C_1(1+\lambda k^{q-p}) \|u_{\lambda,k}\|_{2^{\ast}_s}^{p-2}\|w_{n}\|_{\frac{2r}{r-1}}^{2}\\
				&\leq \beta C_0(1+\lambda k^{q-p}) \|w_{n}\|_{\frac{2r}{r-1}}^{2},
		\end{align*}
	where $r:= 2^{\ast}_s / (p-2)>1 $ and $C_0=\mathbb{S}_{K_s} C_1 \left( \frac{2 \theta c(I_0)}{\mathbb{S}_{ K_s} (\theta -2)} \right)^{\frac{p-2}{2}}$ is a constant independent of $\lambda$ and $k$. Since $|w_{n}|\leq |u_{\lambda,k}|^{\beta}$ in $B_{n}$ and $|w_{n}|=|u_{\lambda,k}|^{\beta}$ in $A_{n}$, this yields
\begin{equation*}
	\left(\int_{A_{n}}|u_{\lambda,k}|^{ 2^{\ast}_{s} \beta}\dx\right)^{\frac{2}{2^{\ast}_{s}}}   \leq \beta C_0 \left(1+\lambda k^{q-p}  \right)\left( \int_{\mathbb{R}^{N}}|u_{\lambda,k}|^{\frac{2r}{r-1} \beta}\dx\right)^{\frac{r-1}{r}}.
\end{equation*}
Passing to the limit via Fatou's Lemma leads to
		\begin{equation}\label{eq:19}
			\|u_{\lambda,k} \|_{  2^{\ast}_{s}\beta   } \leq \left( \beta C_0(1+\lambda k^{q-p})\right)^{\frac{1}{2\beta}}\|u_{\lambda,k}\|_{2\beta r_{0}},
		\end{equation}
		with $r_{0}=  r/(r-1) $.  Next we define  $\eta:= 2^{\ast}_{s} / (2 r_0)  >1.$ Taking $\beta=\eta$ in \eqref{eq:19} and noticing that $2\beta r_{0} =2\eta r_0= 2^{\ast}_{s}$, we have
		\begin{equation}\label{eq:20}
			\|u_{\lambda, k}\|_{2^{\ast}_{s}\eta} \leq \left(\eta C_0(1+\lambda k^{q-p})   \right)^{\frac{1}{2\eta}}\|u_{\lambda,k}\|_{2^{\ast}_{s}}.
		\end{equation}
		In the next step we take $\beta=\eta^{2}$ in \eqref{eq:19} and use \eqref{eq:20},
		\begin{align*}
			\|u_{\lambda,k}\|_{2^{\ast}_{s}\eta^{2}} & \leq \eta^{\frac{1}{\eta^{2}}} \left(C_0(1+\lambda k^{q-p})  \right)^{\frac{1}{2\eta^{2}}}\|u_{\lambda,k}\|_{2^{\ast}_{s}\eta} \\
			& \leq \eta^{\frac{1}{2}\left(  \frac{2}{\eta^{2}}+\frac{1}{\eta} \right)} \left(C_0(1+\lambda k^{q-p})  \right)^{\frac{1}{2}\left(   \frac{1}{\eta^{2}}+\frac{1}{\eta}\right)}\|u_{\lambda,k}\|_{2^{\ast}_{s}}.
		\end{align*}
		Iterating this argument, we conclude that
		\begin{equation}\label{eq4.22}
			\|u_{\lambda,k}\|_{2^{\ast}_{s}\eta^{m}} \leq \eta^{ \frac{1}{2}\left(\frac{1}{\eta}+\frac{2}{\eta^{2}}+\cdots+\frac{m}{\eta^{m}}\right)} \left(C_0(1+\lambda k^{q-p})\right)^{\frac{1}{2}\left(\frac{1}{\eta}+\frac{1}{\eta^{2}}+\cdots+\frac{1}{\eta^{m}}\right)}\|u_{\lambda,k}\|_{2^{\ast}_{s}},
		\end{equation}
		for all $m \in \mathbb{N}$. A simple calculation shows
		\begin{equation*}
			\sum_{i=1}^{\infty}\frac{i}{\eta^{i}}< +\infty\quad \text{and}\quad \sum_{i=1}^{\infty}\frac{1}{\eta^{i}}=\frac{1}{(\eta-1)}.
		\end{equation*}
		Since $\eta > 1$, passing to the limit as $m \to \infty$ in \eqref{eq4.22}, we find
		\begin{equation*}
				\|u_{\lambda,k} \|_{\infty} \leq \eta^{\left(   \sum_{i=1}^{\infty}\frac{i}{2\eta^{i}}   \right) } \left(C_0(1+\lambda k^{q-p})   \right)^{\frac{1}{2(\eta-1)}}\|u_{\lambda,k}\|_{2^{\ast}_{s}} = C \left(1+\lambda k^{q-p} \right)^{\frac{1}{2^{\ast}_s-p}}\|u_{\lambda,k}\|_{2^{\ast}_{s}},
		\end{equation*}
		where $C$ is a constant that does not depend on $\lambda$ or $k$.
	\end{proof}
	
	\begin{corollary}\label{cor45}
		For each $\lambda>0$ and $k \in \mathbb{N}$, we have
		\begin{equation*}
			\|u_{\lambda,k}\|_{\infty}  \leq C_\ast(1+\lambda k^{q-p})^{\frac{1}{2^{\ast}_s-p}}
		\end{equation*}
		for some $C_\ast>0$ independent of $\lambda$ and $k$.
	\end{corollary}
	\begin{proof}
		This is an immediate consequence of the previous result, combined with estimate \eqref{est1} and Proposition \ref{p_conv}.
	\end{proof}
	\section{Proof of Theorem \ref{main}}\label{s_proof}
As previously mentioned, the proof of our main result relies on analyzing a suitable cutoff of the fundamental solution of the fractional Laplacian, which acts as a supersolution for the integro-differential operator \eqref{equ1}. This approach is combined with comparison techniques based on the Maximum Principle from \cite{duarte2} (Lemma \ref{l_max}). We proceed by establishing some key properties of the function $\Gamma_\ast$. First, we note that condition \ref{K_quatro} ensures that $D_{K_s}^{s,2}(\mathbb{R}^N)=D^{s,2}(\mathbb{R}^N)$ and the norms of these spaces are equivalent.
\begin{lemma}\label{l_aste}
$\Gamma_\ast \in D^{s,2}(\mathbb{R}^N).$
\end{lemma}
\begin{proof}
	Due to symmetry, we only need to verify that the contributions to the Gagliardo norm from the domains $B_R \times B_R$, $B_R \times B_R^c$, and $B_R^c \times B_R^c$ are finite. Since $\Gamma_\ast$ is constant in $B_R$, the integral over $B_R \times B_R$ vanishes. Thus, we focus our analysis on the domains $B_R \times B_R^c$ and $B_R^c \times B_R^c$. We start by analyzing the domain $\Omega_1 = (B_R^c \times B_R^c)\cap \mathcal{D},$ where $\mathcal{D}$ is the diagonal $\mathcal{D}=\{ (x,y) \in \mathbb{R}^{2N} : |x-y| \leq |x|/2 \}.$ We have $|y| \geq |x|/2$ for any $(x,y) \in \mathcal{D}.$ Also, for any $(x,y) \in \Omega _1,$ denoting $\gamma := N-2s,$ we have $|\Gamma_\ast (x) - \Gamma_\ast(y)| \leq 2 \gamma |\xi| ^{-\gamma -1} |x-y|,$ for some $\xi \in [x,y],$ the closed line segment connecting $x$ and $y.$ It is known that $\xi =x+ t (y-x), $ for some $t \in [0,1].$ In particular, $|x-\xi | \leq |x-y| \leq |x|/2,$ which leads to $|\xi | \geq |x|/2.$ Thus, for $\sigma := N+2s,$
	\begin{equation}\label{ds1}
		\frac{ (\Gamma_\ast(x) - \Gamma_\ast(y) )^2}{|x-y|^{N+2s}} \leq C |x|^{-2(\gamma +1)} |x-y| ^{2-\sigma},\quad \forall \, (x,y) \in \Omega _1.
	\end{equation}
	On the other hand, if $x \in B_R^c,$ then the change of variables for polar coordinates yields
	\begin{equation}\label{ds2}
		\int _{ B_{ \frac{|x|}{2} } (x)} | x - y | ^{2-\sigma } \dy \leq C |x|^{2(1-s)}.
	\end{equation}
	We now integrate over $\Omega_1$ and use \eqref{ds2} in the resulting integral involving \eqref{ds1}, to conclude that
	\begin{align*}
		I_1:=\iint_{\Omega_1}\frac{ (\Gamma_\ast(x) - \Gamma_\ast(y) )^2}{|x-y|^{N+2s}} \dxdy &\leq C\int _{B^c_R} |x| ^{-2(\gamma +1)} \Big(  \int _{ B_{ \frac{|x|}{2} } (x)} | x - y | ^{2-\sigma } \dy \Big) \dx \\
		&\leq C \int _{B^c_R} |x| ^{-2(\gamma +1) + 2(1-s)} \dx = C R^{-(N-2s)} <+\infty.
	\end{align*}
The estimate over the domain $\Omega _2 = (B^c_R \times B_R^c) \cap \mathcal{D}^c$ starts by looking into the inequality $(\Gamma_\ast(x) - \Gamma_\ast (y) )^2 \leq 2  ( (\Gamma_\ast(x) )^2 + (\Gamma_\ast(y) )^2 ) .$ Hence
\begin{align*}
	I_{21} := \iint _{\Omega _2} \frac{(\Gamma_\ast(x) ) ^2}{|x-y|^{N+2s}}\dxdy &\leq \int _{B^c_R} |x|^{-2 \gamma } \Big( \int _{B^c_{\frac{|x|}{2}}  (x) } |x-y|^{-\sigma }\dy  \Big) \dx\\
	&=C \int _{B_R^c}  |x|^{-2(\gamma +s) } \dx = C R^{-(N-2s)} < + \infty.
\end{align*}
We proceed by considering a fixed $x \in B_R^c$ by denoting $\mathcal{W} _x = B_R^c \cap B^c_{|x|/2}(x)$ and $\mathcal{Y}_x = B_{|x|/2}^c =\{  y \in \mathbb{R}^N : |y| \geq |x|/2\}.$ We have
\begin{align}
I_{22} &= \iint _{\Omega _2}	\frac{(\Gamma_\ast(y) ) ^2}{|x-y|^{N+2s}} \dxdy \nonumber \\
&=  \int _{B^c_R}  \Big( \int _{\mathcal{W} _x \cap  \mathcal{Y}_x} \frac{(\Gamma_\ast(y) ) ^2}{|x-y|^{N+2s}}  \dy \Big) \dx +  \int _{B^c_R}  \Big( \int _{\mathcal{W} _x \cap  \mathcal{Y}^c_x} \frac{(\Gamma_\ast(y) ) ^2}{|x-y|^{N+2s}}  \dy \Big) \dx.\label{ds3}
\end{align}
The first integral on the right-hand side of \eqref{ds3} is estimated as above,
\begin{equation*}
	\int _{B^c_R}  \Big( \int _{\mathcal{W} _x \cap  \mathcal{Y}_x} \frac{(\Gamma_\ast(y) ) ^2}{|x-y|^{N+2s}}  \dy \Big) \dx \leq \int _{B^c_R} |x|^{-2 \gamma } \Big( \int _{B^c_{\frac{|x|}{2}}  (x) } |x-y|^{-\sigma }\dy  \Big) \dx < + \infty.
\end{equation*}
In turn, the second integral in \eqref{ds3} is estimated in the following way
\begin{align}
	\int _{B^c_R}  \Big( \int _{\mathcal{W} _x \cap  \mathcal{Y}^c_x} \frac{(\Gamma_\ast(y) ) ^2}{|x-y|^{N+2s}}  \dy \Big) \dx &\leq C \int _{B_R^c} |x|^{-\sigma } \Big( \int _{B^c_R \cap \mathcal{Y}^c_x} |y|^{-2 \gamma } \dy  \Big) \dx \nonumber\\
	&=C \int _{B_R^c} |x|^{-\sigma } \Big(  \int _R ^{|x|/2}  \varrho ^{-2\gamma +N-1} \drho \Big) \dx.\label{ds4}
\end{align}
The right-hand side of \eqref{ds4} is finite because
\begin{equation*}
	\int _{B^c_R} |x|^{-\sigma} \dx = C \int _R ^\infty \varrho ^{-2s-1} \drho < +\infty\quad \text{and}\quad\int _{B^c_R} |x|^{-\sigma -2 \gamma + N} \dx = \int _R ^\infty \varrho^{-(N-2s) -1} \drho < +\infty.
\end{equation*}
Combining those estimates
\begin{equation*}
	\int _{B_R^c} \int _{B_R^c}\frac{ (\Gamma_\ast(x) - \Gamma_\ast (y) )^2}{|x-y|^{N+2s}}  \dxdy = I_1 + 2(I_{21} + I_{22}) < +\infty.
\end{equation*}
Next, we take into account the set $E = B_{2R}^c = \{  y \in \mathbb{R}^N : |y| \geq 2R   \}.$ The Gagliardo norm estimate over $B_R\times B_R^c$ is made by writing $B_R \times B_R^c = \big(   B_R\times (E \cap B_R ^c )\big)\cup \big(  B_R\times (E^c \cap B_R^c) \big).$ In $B_R\times (E \cap B_R ^c ),$ the term $(\Gamma_\ast(x) - \Gamma_\ast(y) )^2$ is uniformly bounded below and above by positive constants. Moreover, 
\begin{equation*}
	|x-y| >|y|-R\geq |y|- (|y|/2)  = |y|/2 ,\quad \forall \, (x,y) \in B_R\times E .
\end{equation*}
Thus
\begin{align}
	\int_{E \cap B_R ^c } (\Gamma_\ast(x) - \Gamma_\ast(y) )^2  \Big( \int _{B_R} |x-y|^{- \sigma } \dx \Big)  \dy &\leq C \int_{E \cap B_R ^c } \Big( \int _{B_R} |y|^{- \sigma } \dx \Big) \dy \nonumber \\
	& = C \int _{B^c_{2R}} |y|^{-\sigma } \dy =C \int _{2R} ^\infty \varrho ^{-2s -1} \drho < +\infty .\label{ds5}
\end{align}
Nevertheless, for $x\in B_R$ and $y \in E^c \cap \overline{B}_R^c,$ we have $|x-y| > |y| -R>0$ and $\Gamma_\ast(x) = R^{-\gamma }.$ In particular, $B_R \subset \overline{B}_{|y|-R}^c (y) := \{ x \in \mathbb{R}^N : |x-y| >|y|-R \}.$ Consequently, 
\begin{align*}
	\int_{E^c \cap B_R ^c } (\Gamma_\ast(x) - \Gamma_\ast(y) )^2 & \Big( \int _{B_R} |x-y|^{- \sigma } \dx \Big)  \dy \nonumber \\
	& \leq \int_{E^c \cap B_R ^c } (R^{-\gamma } - \Gamma_\ast(y) )^2  \Big( \int _{ \overline{B}_{|y|-R}^c (y) } |x-y|^{- \sigma }\dx \Big) \dy  \\
	&=C \int_{B_{2R} \cap B_R ^c } (R^{-\gamma } - \Gamma_\ast (y) )^2  (|y| - R )^{-2s}  \dy.
\end{align*}
By the Mean Value Theorem applied to the function $\zeta(t) = t^{-\gamma},$ for $t \in [R,2R],$ we obtain $|R^{-\gamma} - |y|^{-\gamma }| \leq \gamma R^{-\gamma - 1} |R - |y||,$ for $y \in B_{2R} \cap B_R ^c.$ Therefore,
\begin{equation}\label{ds7}
	\int_{B_{2R} \cap B_R ^c } (R^{-\gamma } - \Gamma_\ast (y) )^2  (|y| - R )^{-2s}  \dy \leq C \int _{B_{2R} \cap B_R ^c} (|y|-R)^{2(1-s)}\dy <+\infty.
\end{equation}
Summing up, estimates \eqref{ds5}--\eqref{ds7} yield
\begin{equation*}
\int _{B_R} \int _{B_R^c}\frac{ (\Gamma_\ast(x) - \Gamma_\ast (y) )^2}{|x-y|^{N+2s}}  \dxdy < + \infty. \qedhere
\end{equation*}
\end{proof}
We now proceed to establish a general approximation result for a suitable sequence of cutoff functions in $D^{s,2}(\mathbb{R}^N).$ As a preliminary step, we state and prove a standard technical result showing that translations are continuous in $D^{s,2}(\mathbb{R}^N),$ thus making our exposition self-contained.
\begin{lemma}\label{l_trans}
	Let $u \in D^{s,2}(\mathbb{R}^N).$ Then $\|\tau_z u - u\|_{D^{s,2}(\mathbb{R}^N)} \to 0$ as $|z| \to 0,$ where $\tau_z u(x) = u(x-z).$
\end{lemma}
\begin{proof}
	Let $(z_n) \subset \mathbb{R}^N$ be an arbitrary sequence such that $z_n \rightarrow 0.$ Consider $\phi \in C^\infty_0(\mathbb{R}^N).$ We have $\|\tau_{z_n}\phi\|_{D^{s,2}(\mathbb{R}^N)} = \|\phi\|_{D^{s,2}(\mathbb{R}^N)}.$ Moreover, the continuity of $\phi$ implies that $\tau_{z_n}\phi(x) \rightarrow \phi(x)$ for all $x \in \mathbb{R}^N.$ Since the sequence $(\tau_{z_n}\phi)$ is bounded, it converges weakly (up to a subsequence) to some $w \in D^{s,2}(\mathbb{R}^N).$ Due to the continuous embedding $D^{s,2}(\mathbb{R}^N) \hookrightarrow L^{2^\ast_s}(\mathbb{R}^N)$, this weak convergence also holds in $L^{2^\ast_s}(\mathbb{R}^N)$. This fact, combined with the pointwise convergence, identifies the weak limit as $w = \phi.$ The weak convergence $\tau_{z_n}\phi \rightharpoonup \phi$ in $D^{s,2}(\mathbb{R}^N)$ combined with $\|\tau_{z_n}\phi\|_{D^{s,2}(\mathbb{R}^N)} = \|\phi\|_{D^{s,2}(\mathbb{R}^N)}$ implies $\|\tau_{z_n}\phi - \phi\|_{D^{s,2}(\mathbb{R}^N)} \to 0.$ The result for a general $u$ follows by the density of $C^\infty_0(\mathbb{R}^N)$ in $D^{s,2}(\mathbb{R}^N).$
\end{proof}
Our result complements the approximation lemma found in \cite[Lemma 5.3]{dpv} and plays a crucial role in our argument.
\begin{lemma}\label{l_cut}
	Let $u \in D^{s,2}(\mathbb{R}^N)$ and let $\eta \in C^\infty_0(\mathbb{R}^N)$ be a cutoff function such that $0 \leq \eta \leq 1$, $\eta = 1$ in $B_1,$ and $\eta = 0$ in $B^c_2.$ Define $u_n(x) := \eta(x/n)u(x).$ Then $u_n \rightarrow u$ in $D^{s,2}(\mathbb{R}^N).$
\end{lemma}
\begin{proof}
Consider $v_n (x) = u(x) -u_n(x) = (1 - \eta (x/n)) u(x).$ We are going to prove that $(v_n) \subset D^{s,2}(\mathbb{R}^N)$	and $v_n \rightarrow 0$ in $D^{s,2}(\mathbb{R}^N). $ We write 
\begin{equation*}
	v_n(x) - v_n (y) = (u(x) -u(y)) (1- \eta _n (x)) + u(y) (\eta_n(y)  - \eta_n(x)),
\end{equation*}
where $\eta_n  = \eta (\cdot /n).$ Therefore, we have the following estimate
\begin{align}
	\| v_n \|_{D^{s,2}(\mathbb{R}^N)}^2 \leq 2 \int _{\mathbb{R}^N }\int _{\mathbb{R}^N }\frac{(u(x) -u(y))^2 (1- \eta _n (x))^2}{|x-y|^{N+2s}}  \dxdy \nonumber \\+ 2 \int _{\mathbb{R}^N }\int _{\mathbb{R}^N }\frac{ (u(y))^2 (\eta_n(x)  - \eta_n(y))^2 }{|x-y|^{N+2s}}  \dxdy \label{leminha1}
\end{align}
By the Lebesgue convergence theorem, the first integral (denoted by $J_1(n)$) in the right-hand side of \eqref{leminha1} goes to zero, as $n \rightarrow \infty.$ Denoting the second integral by $J_2(n)$, we can write it as
\begin{equation}\label{jota2}
	J_2(n) = \int _{\mathbb{R}^N} (u(y))^2 \Big( \int _{\mathbb{R}^N} \frac{( \eta_n(x) - \eta_n(y) )^2}{|x-y|^{N+2s}}   \dx  \Big) \dy.
\end{equation}
Next, we analyze the second integral of \eqref{jota2}. By a change of variables,
\begin{equation*}
	J_3(y,n) := \int _{\mathbb{R}^N} \frac{( \eta_n(x) - \eta_n(y) )^2}{|x-y|^{N+2s}}   \dx=n^{-2s} \int _{\mathbb{R}^N }   \frac{( \eta (z + (y/n)) -\eta (y/n) )^2  }{|z|^{N+2s}}\dz.
\end{equation*}
Nevertheless,
\begin{equation*}
	\int _{B_1 }   \frac{( \eta (z + (y/n)) -\eta (y/n) )^2  }{|z|^{N+2s}}\dz \leq \| \nabla \eta \|_\infty^2 \int _{B_1} |z|^{-(N+2s) + 2} \dz = C_N \frac{\| \nabla \eta \|_\infty ^2}{2(1-s)}.
\end{equation*}
Furthermore,
\begin{equation*}
	\int _{B^c_1 }   \frac{( \eta (z + (y/n)) -\eta (y/n) )^2  }{|z|^{N+2s}}\dz \leq 2 \| \eta \|_\infty ^2 \int _{B^c_1 }  |z|^{-(N+2s)}  \dz = C_N \frac{\| \eta \|_\infty ^2}{2s}.
\end{equation*}
Summing up, $J_3(y,n) \leq C_N n^{-2s},$ for some $C_N>0$ independent of $y$ and $n.$ Now, observe that the difference function $\mu(x,y):=\eta_n(x) - \eta_n(y)$ is zero whenever $(x,y)$ belongs to $(B_n \times B_n)$ or $B^c_{2n} \times B^c_{2n}.$ Introducing the notation $A_n = \{ z \in \mathbb{R}^N : n \leq |z| < 2n \},$ we have
\begin{align}
	\supp(\mu )  &\subset \left(  (B_n \times B_n) \cup (B^c_{2n} \times B^c_{2n} ) \right) ^c \nonumber\\
	&= (A_n \times \mathbb{R}^N ) \cup (\mathbb{R}^N  \times A_n) \cup  (B^c_{2n} \times B_n) \cup (B_n \times B^c_{2n}).\label{setao}
\end{align}
With the aid of the uniform bound for $J_3$, we estimate the contribution to $J_2(n)$ from each subset on the right-hand side of \eqref{setao}. For the integral over $A_n \times \mathbb{R}^N,$ we apply H\"{o}lder's inequality as follows
\begin{align*}
	\int_{A_n} \int_{\mathbb{R}^N} \frac{ (u(y))^2 (\eta_n(x)  - \eta_n(y))^2 }{|x-y|^{N+2s}}  \dxdy &\leq \int _{A_n} (u(y))^2 J_3(y,n) \dy \\
	&\leq C_N n^{-2s} \int _{A_n} u(y)^2 \dy \\
	&\leq C_N n^{-2s}  |A_n|^{2s/N} \| u \|^2_{L^{2^\ast _s}(A_n)} = C_N \| u \|^2_{L^{2^\ast _s}(A_n)}.
\end{align*}
Since $u \in L^{2^\ast_s}(\mathbb{R}^N)$, this term vanishes as $n \to \infty$. By symmetry, the integral over $\mathbb{R}^N \times A_n$ also goes to zero as $n \rightarrow \infty.$ The next step is analyzing the integral over $B^c_{2n} \times B_n.$ In this case, $(\eta_n(x) - \eta_n(y))^2 = 1$. Moreover, we make use of the elementary inequality $(u(y))^2 \leq 2 (u(x)) ^2 + 2(u(x) - u (y))^2$ to obtain
\begin{multline}\label{gagli}
	\iint_{B^c_{2n} \times B_n} \frac{ (u(y))^2 (\eta_n(x)  - \eta_n(y))^2 }{|x-y|^{N+2s}}  \dx\dy \\
	\leq 2 \iint_{B^c_{2n} \times B_n} \frac{(u(x)-u(y))^2}{|x-y|^{N+2s}} \dx\dy + 2 \int_{B_{2n}^c} (u(x))^2 \Big( \int_{B_n} \frac{1}{|x-y|^{N+2s}} \dy \Big) \dx.
\end{multline}
The first integral on the right-hand side of \eqref{gagli} goes to zero as $n \to \infty,$ because the integrand is integrable over $\mathbb{R}^N \times \mathbb{R}^N$, since $u \in D^{s,2}(\mathbb{R}^N).$ For the second integral, since $x \in B_{2n}^c$ and $y \in B_n$, we have $|x-y| \geq |x|/2$. Thus, $|x-y|^{- (N+2s)}\leq C_N n^{-N} |x|^{-2s}$ and the inner integral is bounded by $C_N|x|^{-2s}$, yielding
\begin{equation}\label{hardy}
	\int_{B_{2n}^c} (u(x))^2 \Big( \int_{B_n} \frac{1}{|x-y|^{N+2s}} \dy \Big) \dx \leq C_N\int_{B_{2n}^c} \frac{u(x)^2}{|x|^{2s}} \dx.
\end{equation}
By the fractional Hardy inequality (see \cite{zbMATH01374984}), $u \in D^{s,2}(\mathbb{R}^N)$ implies $\int_{\mathbb{R}^N} u^2 |x|^{-2s} \dx < +\infty$. Hence, the right-hand side of \eqref{hardy} converges to zero as $n \rightarrow \infty.$ The remaining integral is estimated using a similar approach. When $(x,y) \in B_n \times B^c_{2n}$, we also have $(\eta_n(x) - \eta_n(y))^2 = 1$. In this case, we can write
\begin{equation*}
	\int_{B_n } \int_{B^c_{2n}}  \frac{ (u(y))^2 (\eta_n(x)  - \eta_n(y))^2 }{|x-y|^{N+2s}}  \dx\dy  = \int _{B^c_{2n} } (u(y))^2 \Big( \int _{B_n} |x-y|^{-(N+2s)} \dx \Big) \dy,
\end{equation*}
which corresponds exactly to the left-hand side of \eqref{hardy} with the variables interchanged. Thus, the integral over $B_n \times B^c_{2n}$ also vanishes as $n \rightarrow \infty.$ Combining all these estimates, we conclude that $J_2(n) \to 0$ as $n \to \infty$, which completes the proof.
\end{proof}
The preceding technical results allow us to establish the following density property. It guarantees that any function in $D^{s,2}(\mathbb{R}^N)$ vanishing in a ball can be approximated in the $D^{s,2}(\mathbb{R}^N)$ norm by test functions that also vanish in the same ball.
\begin{proposition}\label{p_dense}
	Suppose $u \in D^{s,2}(\mathbb{R}^N)$ satisfies $u = 0$ a.e. in $B_R.$ Then, there exists a sequence $(\phi_n) \subset C^\infty_0(\mathbb{R}^N \setminus \overline{B}_R)$ such that $\phi_n \to u$ in $D^{s,2}(\mathbb{R}^N),$ with the additional property that $\phi_n \geq 0,$ whenever $u \geq 0$ a.e. in $\mathbb{R}^N$.
\end{proposition}
\begin{proof}
	Let $\eta \in C^\infty_0(\mathbb{R}^N)$ and $u_n = \eta(\cdot/n) u$ as in Lemma \ref{l_cut}. For all sufficiently large $n$, it follows that $\supp(u_n) \subset B_{2n} \setminus B_R.$ Moreover, Lemma \ref{l_cut} also guarantees that $u_n \rightarrow u$ in $D^{s,2}(\mathbb{R} ^N).$ Hence, for a given $\varepsilon >0,$ there is $n_0 \in \mathbb{N}$ such that
	\begin{equation}\label{epis1}
		\| u_n - u \| _{D^{s,2}(\mathbb{R} ^N)} < \varepsilon/3,\quad  \text{for }n \geq n_0.
	\end{equation}
	 Next, for $\sigma > 1$ and $n$ as above, we define $u_{n,\sigma } = u_n (\cdot / \sigma ).$ Since $\| u_{n,\sigma } \|_{D^{s,2}(\mathbb{R}^N)}^2 = \sigma ^{N-2s} \| u_{n} \|_{D^{s,2}(\mathbb{R}^N)}^2, $ an argument similar to that used in Lemma \ref{l_trans} shows that $u_{n, \sigma } \rightarrow u_n$ in $D^{s,2}(\mathbb{R}^N),$ as $\sigma \rightarrow 1^+.$ In particular, there is $\sigma _0 = \sigma_0(n) >1$ such that 
	 \begin{equation}\label{epis2}
	 	\| u_{n,\sigma } - u_n \| _{D^{s,2}(\mathbb{R} ^N)} < \varepsilon/3,\quad  \text{whenever }1<\sigma  \leq \sigma_0.
	 \end{equation}
	 Next, we consider the standard mollifier sequence given by $\varphi _m (x) = (1 / (1/m)^N )\varphi (x/ (1/m)),$ where $\varphi(x) = C \exp\{ 1/(|x|^{2} -1 )  \},$ if $|x|<1,$ and $\varphi(x) = 0,$ if $|x| \geq 1,$ with $C>0$ being a suitable normalizing constant such that $\int _{\mathbb{R}^N} \varphi _m\dz=1$. We claim that the mollified sequence $u_{n,\sigma,m} := u_{n,\sigma } \ast \varphi _m \in C^\infty (\mathbb{R}^N)$ provides the desired approximation. To prove this, we start by pointing out that $u_{n,\sigma,m}$ has compact support with $\supp{ (u_{n,\sigma,m} ) } \subset \overline{(B_{2n\sigma} \setminus B_{\sigma R} )  + B_{1/m} }.$ Using this fact, one can check that for $m$ sufficiently large with $(1/m) < (\sigma -1)R,$ we have: if $|x| < \sigma R - (1/m),$ then $u_{n,\sigma,m} (x) = 0.$ Equivalently, $\supp{ (u_{n,\sigma,m} ) } \subset B^c_{\sigma R - (1/m)}$ and $B_R \varsubsetneq B_{\sigma R - (1/m)},$ for $m$ large enough. Consequently, fixing this large $m,$ one has $u_{n,\sigma,m} \in  C^\infty_0(\mathbb{R}^N \setminus \overline{B}_R) .$ We proceed by writing
	\begin{equation*}
		u_{n,\sigma,m} (x) - u_{n,\sigma }(x) = \int _{B_{1/m}} \big(  \tau _{z} u_{n,\sigma }(x) - u_{n,\sigma }(x) \big) \varphi _{m}(z)\dz.
	\end{equation*}
	By denoting 
	\begin{equation*}
	u_\ast(x,y,z) = (\tau _{z} u_{n,\sigma }(x) - u_{n,\sigma }(x)) - (\tau _{z} u_{n,\sigma }(y) - u_{n,\sigma }(y)),
	\end{equation*}
	one can write
	\begin{equation*}
		\| u_{n,\sigma,m}    - u_{n,\sigma}\|_{D^{s,2}(\mathbb{R}^N)} = \left(  \iint _{\mathbb{R}^{2N}} \Big( \int _{B_{1/m}}  u_\ast (x,y,z) \varphi _{m}(z) \dz  \Big) ^2 |x-y|^{-(N+2s)}\dxdy \right) ^{1/2}.
	\end{equation*}
At this point, since the measure $\mu_s (W) = \iint _W |x-y|^{-(N+2s)} \dxdy$ is $\sigma$-finite, we can apply Minkowski's inequality for integrals (see \cite[Appendix A]{zbMATH03441026}) to get the following estimate,
\begin{align*}
	\| u_{n,\sigma,m}    - u_{n,\sigma}\|_{D^{s,2}(\mathbb{R}^N)} &\leq \int _{B_{1/m}} \Big( \iint _{\mathbb{R}^{2N}} \big(u_\ast(x,y,z) \varphi _m (z)  \big)^2  |x-y|^{-(N+2s)}  \dxdy \Big) ^{1/2}  \dz\\
	&=\int _{B_{1/m} }  \big(\|\tau _z u_{n,\sigma } - u_{n,\sigma } \|_{D^{s,2}(\mathbb{R}^N)} \big) \varphi_m (z) \dz .
\end{align*}
In the next step we apply Lemma \ref{l_trans}. There is $\delta = \delta (n,\sigma) >0$ such that $\|\tau_z u_{n,\sigma } - u_{n,\sigma }\|_{D^{s,2}(\mathbb{R}^N)} < \varepsilon/3 ,$ provided that $|z| < \delta.$ Taking a larger $m_0 = m_0 (n,\sigma) $ (if necessary) with $1/m_0 < \delta$ yields
\begin{equation}\label{epis3}
	\| u_{n,\sigma,m_0}    - u_{n,\sigma}\|_{D^{s,2}(\mathbb{R}^N)} \leq \int _{B_{1/m_0} }  \big(\|\tau _z u_{n,\sigma } - u_{n,\sigma } \|_{D^{s,2}(\mathbb{R}^N)} \big) \varphi_{m_0} (z) \dz < \varepsilon /3.
\end{equation}
Therefore, the choice of a suitably large $m_0 = m_0(n,\sigma)$ leads directly to \eqref{epis3}. Let us fix $n_0$ and $\sigma_0$ as in \eqref{epis1} and \eqref{epis2}, respectively, and take $m\geq m_0$ large enough so that $0 < 1/m <  (\sigma_0 -1)R.$ We obtain
\begin{equation*}
\|	u_{n_0,\sigma_0,m} - u \|_{D^{s,2}(\mathbb{R}^N)} \leq \| u_{n_0,\sigma_0,m}    - u_{n_0,\sigma_0}\|_{D^{s,2}(\mathbb{R}^N)} + \| u_{n_0,\sigma_0 } - u_{n_0} \| _{D^{s,2}(\mathbb{R} ^N)} + \| u_{n_0} - u \| _{D^{s,2}(\mathbb{R} ^N)}<\varepsilon.
\end{equation*}
In conclusion, for any $\varepsilon > 0$, we have found a function $\phi_\varepsilon := u_{n_0,\sigma_0,m} \in C^\infty_0(\mathbb{R}^N \setminus \overline{B}_R)$ such that $\| \phi_\varepsilon - u \|_{D^{s,2}(\mathbb{R}^N)} < \varepsilon$, which establishes the density property and completes the proof.
\end{proof}
Combining the uniform bound from Corollary \ref{cor45}, we obtain a precise domination of the solution $u_{\lambda,k}$ by $\Gamma_\ast$ near the origin, more precisely,
\begin{equation}\label{est_orig}
	u_{\lambda,k}(x)\leq C_\ast   R^{N-2s}  (1+\lambda k^{q-p})^{\frac{1}{2^{\ast}_s-p}}\Gamma_\ast (x),\quad \text{for } |x| <R.
\end{equation}
Having established the estimate inside the ball, we now turn our attention to the decay behavior of the solution in the exterior region, proving that it decays polynomially at infinity.
	\begin{proposition}\label{p_brutal}
		The positive solution $u_{\lambda,k}$ of the auxiliary problem satisfies
		\begin{equation*}
		u_{\lambda,k}(x)\leq C_\ast   R^{N-2s} (1+\lambda k^{q-p})^{\frac{1}{2^{\ast}_s-p}} |x|^{-(N-2s)},\quad \text{for }|x|>R.
		\end{equation*}
	\end{proposition}
	\begin{proof}
		Let us define the function
		\begin{equation*}
			v(x) := C_\ast R^{N-2s} (1+\lambda k^{q-p})^{\frac{1}{2^{\ast}_s-p}} \Gamma_\ast(x),
		\end{equation*}
		and set $w := u_{\lambda,k} - v.$ By estimate \eqref{est_orig}, we have $u_{\lambda,k}(x) \leq v(x)$ for $|x| < R,$ which implies $w^+ = 0$ in $B_R.$ Our goal is to show that $w^+ = 0$ in $\mathbb{R}^N.$
		To this end, Lemma \ref{l_aste} ensures $v,w \in D^{s,2}(\mathbb{R}^N)$. Consequently, Proposition \ref{p_dense} allows us to find an approximating sequence $(\phi_n) \subset C^\infty_0(\mathbb{R}^N \setminus \overline{B}_R)$ such that $\phi_n \to w^+$ in $D^{s,2}(\mathbb{R}^N).$ Next, employing the elementary inequality $(a^+ - b^+)^2 \leq (a-b)(a^+ - b^+)$ for $a,b \in \mathbb{R},$ we obtain
		\begin{equation*}
			[w^+]^2 \leq [w, w^+] = \lim _{n \rightarrow \infty } [w,\phi_n ] = \lim _{n \rightarrow \infty } \big( [u_{\lambda,k},\phi_n] - [v ,\phi_n] \big).
		\end{equation*}
		On the other hand, condition \ref{K_cinco} implies that $[\Gamma_\ast ,\phi_n] \geq 0.$ Since $v$ is a positive multiple of $\Gamma_\ast,$ it follows that
		\begin{equation*}
			[w^+]^2 \leq \lim _{n \rightarrow \infty } [u_{\lambda,k},\phi_n] = [u_{\lambda,k} , w^+].
		\end{equation*}
		Furthermore, using another elementary inequality, $(a-b)^+ \leq a$ for $a,b \geq 0,$ yields $w^+ \in E.$ This allows us to take $w^+$ as a test function in \eqref{def_ws}. Taking Remark \ref{r1} into account, we find
		\begin{equation*}
			[u_{\lambda,k}, w^+] = \int _{B_R^c } g(x,u_{\lambda,k} ) w^+ \dx - \int _{B_R^c }  V(x) u_{\lambda,k}  w^+ \dx \leq \left( \frac{1}{\nu} -1 \right) \int _{B_R^c} V(x) u_{\lambda,k}  w^+ \dx \leq 0.
		\end{equation*}
		Consequently, $[w^+]^2 \leq 0,$ which implies $w^+ = 0,$ completing the proof.
	\end{proof}
	Equipped with the necessary preceding results, we are now ready to prove our main result.	
	\begin{proof}[Proof of Theorem \ref{main} completed]
		By Corollary \ref{cor45},
		\begin{equation*}
			\|u_{\lambda,k}\|_{\infty}  \leq C_\ast(1+\lambda k^{q-p})^{\frac{1}{2^{\ast}_s-p}}.
		\end{equation*}
		Observe that the quantity $(1+\lambda k^{q-p})^{\frac{1}{2^{\ast}_s-p}}$ can be made arbitrarily close to $1$ by taking $\lambda$ sufficiently small. Specifically, fixing $k_0 > C_\ast$, there exists $\lambda_0 > 0$ such that if $\lambda \in (0, \lambda_0)$, then $(1+\lambda k_0^{q-p})^{\frac{1}{2^{\ast}_s-p}} < k_0 / C_\ast$. Thus, $\|u_{\lambda,k_0}\|_{\infty} < k_0,$ for $\lambda \in (0, \lambda_0)$. Consequently, by the definition of $f_{\lambda, k_0}$, we have
		\begin{equation*}
			f_{\lambda, k_0}(u_{\lambda,k_0}) = f(u_{\lambda,k_0})+\lambda u_{\lambda,k_0}^{q-1},\quad \text{for }\lambda \in (0, \lambda_0).
		\end{equation*}
		On the other hand, Remark \ref{r1} yields the bound $f_{\lambda , k_0} (t) \leq C_1 (1+\lambda_0 k_0^{q-p} ) t^{2_s^\ast-1} = C'_1 t^{2_s^\ast-1}.$ Combining this with Proposition \ref{p_brutal} allows us to deduce the following estimate
		\begin{equation*}
			\frac{f_{\lambda, k_0} (u_{\lambda,k_0}) }{u_{\lambda,k_0}} \leq C'_1 C_\ast ^{2^\ast _s -2} (1+\lambda _0 k_0^{q-p} )^{\frac{2^\ast_s -2}{2^\ast _s - p }} R^{(N-2s)(2^\ast _s -2)  } |x|^{- (N-2s)(2^\ast _s -2)},\quad \text{for }|x| > R.
		\end{equation*}
	Next, let us set $\Lambda _\ast = \nu  C'_1 C_\ast ^{2^\ast _s -2} (1+\lambda _0 k_0^{q-p} )^{\frac{2^\ast_s -2}{2^\ast _s - p }}.$ For any $\Lambda > \Lambda_\ast$, it readily follows that $\Lambda / \nu > C'_1 C_\ast ^{2^\ast _s -2} (1+\lambda _0 k_0^{q-p} )^{\frac{2^\ast_s -2}{2^\ast _s - p }}.$ Therefore, recalling that $(N-2s)(2^\ast_s - 2) = 4s$, using condition \ref{V_dois}, we obtain
	\begin{equation*}
	\frac{f_{\lambda, k_0} (u_{\lambda,k_0}) }{u_{\lambda,k_0}} \leq \frac{1}{\nu} \Lambda \frac{R^{4s}}{|x|^{4s}} \leq \frac{1}{\nu}V(x),\quad \text{for }|x|>R.
	\end{equation*}
	This implies that $\bar{f}_{\lambda,k_0} (u_{\lambda,k_0}) = f_{\lambda,k_0} (u_{\lambda,k_0}),$ for $|x| > R.$ In particular, for any $\lambda \in (0,\lambda_0)$, the modified nonlinearity satisfies $g_{\lambda , k_0} (u_{\lambda,k_0}) = f(u_{\lambda,k_0})+\lambda u_{\lambda,k_0}^{q-1}. $ Thus, setting $u_{\lambda, \Lambda} := u_{\lambda, k_0}$, we conclude that $u_{\lambda, \Lambda}\in E$ is a solution to Eq. \eqref{P} provided $\lambda \in (0,\lambda _0)$ and $\Lambda > \Lambda _\ast$.
	\end{proof}

	\appendix
	
	\section{A nontrivial example of a singular kernel satisfying our hypotheses}\label{s_app_ex}
	The purpose of this appendix is to show that the kernel $K_s(x) = (1 + a(x))|x|^{-(N+2s)}$ verifies hypotheses \ref{K_um}--\ref{K_cinco}, where $a \in C^\infty(\mathbb{R}^N)$ is a radial cut-off function satisfying $0 \leq a(x) \leq 1$ in $\mathbb{R}^N$, with $a(x) = 0$ for $|x| \leq R/2$ and $a(x) = 1$ for $|x| \geq R$. Since $a$ is radial and bounded, it is clear that $K_s$ readily satisfies \ref{K_um}--\ref{K_quatro}. In particular, $D^{s,2}(\mathbb{R}^N) = D_{K_s}^{s,2}(\mathbb{R}^N).$ Thus, it only remains to prove \ref{K_cinco}. To do so, we recall Lemma \ref{l_aste} and begin by establishing the following fact.
	\begin{proposition}\label{p_supersol}
		$(-\Delta)^s \Gamma_\ast \geq 0$ in $\mathbb{R}^N \setminus \overline{B}_R$ in the weak sense.
	\end{proposition}
	\begin{proof}
	The proof relies on the decomposition $\Gamma _\ast = \Gamma + (\Gamma _\ast - \Gamma)$. Although $\Gamma \notin D^{s,2}(\mathbb{R}^N)$, the bilinear form $[\Gamma, \phi]_{D^{s,2}(\mathbb{R}^N)}$ is well defined and its integrability is rigorously established in \cite{MR5039745}. In particular, $[\Gamma _\ast - \Gamma , \phi]_{D^{s,2}(\mathbb{R}^N)}$ is also well defined. Since $(\Gamma _\ast - \Gamma) (x) = 0$ for $|x| > R$, and $(\Gamma _\ast - \Gamma) (x) = R^{-(N-2s)} - |x|^{-(N-2s)}$ for $|x| \leq R$, a direct computation allows us to write
	\begin{multline*}
		\int_{\mathbb{R}^N } \int_{\mathbb{R}^N } \big( (\Gamma _\ast - \Gamma) (x) - (\Gamma _\ast - \Gamma) (y)  \big) (\phi (x) - \phi(y)) |x-y|^{-(N+2s)}   \dxdy  \\
		= 2 \int _{B_R} \int _{B_R^c } \big( |x|^{-(N-2s)} - R^{-(N-2s)}  \big)  \phi (y) |x-y|^{-(N+2s)} \dxdy \geq 0,
	\end{multline*}
	for all $\phi \in C^\infty _0( \mathbb{R}^N \setminus \overline{B}_R)$ with $\phi \geq 0$. Thus, using the fact that $(-\Delta)^s \Gamma = 0$ weakly in $\mathbb{R}^N\setminus \{ 0\}$ (see also \cite{MR5039745}), we conclude that
	\begin{align*}
		[\Gamma _\ast ,\phi ]_{D^{s,2} (\mathbb{R}^N )} &= [ \Gamma ,\phi]_{D^{s,2} (\mathbb{R}^N )} + [   \Gamma _\ast - \Gamma  ,\phi ]_{D^{s,2} (\mathbb{R}^N )} \\
		&= [   \Gamma _\ast - \Gamma  ,\phi ]_{D^{s,2} (\mathbb{R}^N )} \geq 0,\quad \forall \, \phi \in C^\infty _0( \mathbb{R}^N \setminus \overline{B}_R),\ \phi \geq 0. \qedhere
	\end{align*}
	\end{proof}
		
	The remainder of this appendix aims to show that $[\Gamma_\ast,\phi] \geq 0$ for any nonnegative test function $\phi \in C^\infty _0( \mathbb{R}^N \setminus \overline{B}_R)$. Let us define the auxiliary kernel $K_a(x) = a(x) |x|^{-(N+2s)}$. Since $[\Gamma_\ast ,  \phi ] _{D^{s,2}(\mathbb{R}^N)} \geq 0$ by Proposition \ref{p_supersol}, we deduce that
	\begin{equation}\label{apx_um}
		[\Gamma_\ast ,  \phi ] \geq \int _{\mathbb{R}^N} \int _{\mathbb{R}^N} ( \Gamma _\ast (x) - \Gamma _\ast (y) ) (\phi (x) - \phi(y)) K_a(x-y) \dxdy.
	\end{equation}
	Therefore, it suffices to show that the right-hand side of \eqref{apx_um} is nonnegative. We point out that, by removing the singularity of the fractional Laplacian kernel near the origin, the analysis of this term can be performed much more directly, as the following result shows.
	\begin{lemma}\label{l_integra}
		The functions $(x,y) \mapsto \phi (x) (\Gamma  (x)   - \Gamma  (y) ) K_a(x-y) ,$ $(x,y) \mapsto \phi (x) (\Gamma _\ast (x)   - \Gamma _\ast (y) ) K_a(x-y) $ and $(x,y) \mapsto \phi (x) (\Gamma  (x)   - \Gamma  (y) ) |x-y|^{-(N+2s)} $ belong to $L^1 (\mathbb{R}^N \times \mathbb{R}^N).$
	\end{lemma}
	\begin{proof}
		By the Fubini-Tonelli theorem, since $\supp( \phi) \subset \mathbb{R}^N \setminus \overline{B}_R$ and $K_a(x-y) = 0$ for $|x-y| \leq R/2$, it suffices to prove that
		\begin{equation*}
			\int _{\mathbb{R}^{N} } \int _{\mathbb{R}^{N} } \phi (x) \big|\Gamma  (x)   - \Gamma  (y) \big| K_a(x-y)  \dxdy = \int _{B_R^c } \phi (x) \left(  \int _{B^c_{R/2} (x) } \big|\Gamma  (x)   - \Gamma  (y) \big| K_a(x-y) \dy   \right)   \dx  < + \infty.
		\end{equation*}
		To estimate the inner integral for a fixed $x \in \supp(\phi) \subset B_R^c$, we split the domain of integration into $B_R$ and $B_R^c$. For $y \in B^c_{R/2}(x) \cap B_R$, we use the bounds $K_a(x-y) \leq (R/2)^{-(N+2s)}$ and $\Gamma(x) \leq R^{-(N-2s)},$ together with the integrability of $\Gamma$ over $B_R$, to obtain the following estimate
		\begin{equation}\label{fb_um}
			\int _{B^c_{R/2}(x) \cap B_R   } \big| \Gamma  (x) - \Gamma  (y)  \big| K_a(x-y)\dy \leq C_N R^{-N} .  
		\end{equation}
		Next, for $y \in B^c_{R/2}(x) \cap B^c_R$, we have $\Gamma(x) \leq R^{-(N-2s)}$ and $\Gamma(y) \leq R^{-(N-2s)}$. Hence, using the bound $K_a(z) \leq |z|^{-(N+2s)}$ and setting $z = x-y$ yields
		\begin{equation}\label{fb_dois}
			\int _{B^c_{R/2}(x) \cap B^c_R   }\big| \Gamma  (x) - \Gamma  (y)  \big| K_a(x-y)\dy \leq 2 R^{-(N-2s)} \int _{B^c_{R/2}} |z|^{-(N+2s)} \dz   \leq C_N R^{-N}.
		\end{equation}
		Combining \eqref{fb_um} and \eqref{fb_dois}, we conclude that
		\begin{equation*}
			\int _{B_R^c } \phi (x) \left(  \int _{B^c_{R/2} (x) } \big|\Gamma  (x)   - \Gamma  (y) \big| K_a(x-y) \dy   \right)   \dx \leq C_N R^{-N} \int _{B_R^c } \phi \dx < + \infty.
		\end{equation*}
		The integrability argument for the function $(x,y) \mapsto \phi (x) (\Gamma _\ast (x)   - \Gamma _\ast (y) ) K_a(x-y)$ is analogous. The fact that $(x,y) \mapsto \phi (x) (\Gamma  (x)   - \Gamma  (y) ) |x-y|^{-(N+2s)} $ belongs to $L^1 (\mathbb{R}^N \times \mathbb{R}^N)$ is established in \cite{MR5039745}.
	\end{proof}
Lemma \ref{l_integra}, guarantees sufficient integrability for the right-hand side of \eqref{apx_um}. This allows us to use the symmetry of the kernel to rewrite the integral as follows
	\begin{equation*}
		\int _{\mathbb{R}^N} \int _{\mathbb{R}^N} ( \Gamma _\ast (x) - \Gamma _\ast (y) ) (\phi (x) - \phi(y)) K_a(x-y) \dxdy = 2 \int _{B_R^c }   \phi(x) \left(  \int _{\mathbb{R}^N} (\Gamma_\ast (x) - \Gamma_\ast (y) ) K_a(x-y)  \dy  \right) \dx.
	\end{equation*}
	Moreover, for $x \in B^c_R,$ we have
	\begin{equation*}
		\Gamma_\ast (x) - \Gamma_\ast (y) = (\Gamma (x ) - \Gamma (y)) + (\Gamma (y) - \Gamma _\ast (y)) \geq \Gamma (x ) - \Gamma (y).
	\end{equation*}
	Thus, for the same reason as above,
	\begin{equation*}
		\begin{aligned}
			\int _{\mathbb{R}^N} \int _{\mathbb{R}^N} ( \Gamma _\ast (x) - \Gamma _\ast (y) ) (\phi (x) - \phi(y)) &K_a(x-y) \dxdy  \\ &\geq 2 \int _{B_R^c }   \phi(x) \left(  \int _{\mathbb{R}^N} (\Gamma (x) - \Gamma (y) ) K_a(x-y)  \dy  \right) \dx\\
			&=\int _{\mathbb{R}^N} \int _{\mathbb{R}^N} ( \Gamma  (x) - \Gamma (y) ) (\phi (x) - \phi(y)) K_a(x-y) \dxdy.
		\end{aligned}
	\end{equation*}
	On the other hand, we can decompose $K_a(z) = |z|^{-(N+2s)}-\mathcal{A}(z)$, where $\mathcal{A}(z) = (1-a(z))|z|^{-(N+2s)}$. Since $\Gamma$ is the fundamental solution of the fractional Laplacian in $\mathbb{R}^N \setminus \{ 0 \}$, the integral associated with the standard kernel vanishes, yielding
	\begin{equation*}
		\int _{\mathbb{R}^N} \int _{\mathbb{R}^N} ( \Gamma  (x) - \Gamma (y) ) (\phi (x) - \phi(y)) K_a(x-y) \dxdy = \int _{\mathbb{R}^N} \int _{\mathbb{R}^N}   ( \Gamma  (y) - \Gamma (x) )(\phi (x) - \phi(y)) \mathcal{A}(x-y) \dxdy.
	\end{equation*}
	We once again apply Lemma \ref{l_integra} to write
	\begin{equation}\label{int_hard}
		\int _{\mathbb{R}^N} \int _{\mathbb{R}^N}   ( \Gamma  (y) - \Gamma (x) )(\phi (x) - \phi(y)) \mathcal{A}(x-y) \dxdy = 2 \int _{\mathbb{R}^N} \phi (x) \left( \int _{\mathbb{R}^N}  ( \Gamma  (y) - \Gamma (x) ) \mathcal{A}(y-x) \dy \right)  \dx .
	\end{equation}
	Since $\phi$ is supported outside $\overline{B}_R$, we restrict our analysis of the inner integral on the right-hand side of \eqref{int_hard} to the case $x \in B_R^c$. Considering the change of variables $z = y - x$, we find
	\begin{equation*}
		\int _{\mathbb{R}^N}  ( \Gamma  (y) - \Gamma (x) ) \mathcal{A}(y-x) \dy = \int _{\mathbb{R}^N}  ( \Gamma  (x+z) - \Gamma (x) ) \mathcal{A}(z) \dz.
	\end{equation*}
	Next, we pass to polar coordinates $z = \varrho \omega$, with $\varrho >0$ and $\omega \in \mathbb{S}^{N-1}$. Recall that $\operatorname{supp} (\mathcal{A}) \subset \overline{B}_R$. Denoting by $\mathcal{A}_\ast $ the radial profile of $\mathcal{A}$, so that $\mathcal{A}(z) = \mathcal{A}_\ast (|z|)$, we can rewrite the integral as
	\begin{align}
		\int _{\mathbb{R}^N}  ( \Gamma  (x+z) - \Gamma (x) ) & \mathcal{A}(z) \dz = \int _0 ^R \left( \int _{\mathbb{S}^{N-1} } ( \Gamma  (x+\varrho \omega ) - \Gamma (x) ) \mathcal{A}(\varrho \omega)   \varrho ^{N-1}     \,{\rm d}S_{\omega } \right)  \drho \nonumber \\
		&=N |B_1|  \int _0 ^R \mathcal{A}_\ast (\varrho ) \varrho ^{N-1}\left( \frac{1}{N|B_1| \varrho ^{N-1}}\int _{\partial B_\varrho (x) }  (\Gamma (\xi )  -\Gamma (x) ) \,{\rm d}S_{\xi } \right) \drho,\label{qse_fim}
	\end{align}
	where we applied the change of variables $\xi = x + \varrho \omega$ in the surface integral. Observe that $\Delta \Gamma (x) = 2(1-s)(N-2s)|x|^{-(N-2s) - 2} >0$ for $x \neq 0$. Since $|x| > R \geq \varrho,$ this strict subharmonicity guarantees the classical mean value inequality
	\begin{equation*}
		\Gamma (x) \leq \frac{1}{N|B_1| \varrho ^{N-1}}\int _{\partial B_\varrho (x) }  \Gamma (\xi )  \,{\rm d}S_{\xi },\quad\text{for } 0<\varrho \leq R.
	\end{equation*}
	This inequality implies that the expression in \eqref{qse_fim} is nonnegative. Following the chain of inequalities above, this leads to the conclusion that the right-hand side of \eqref{apx_um} is also nonnegative. Since the nonnegative test function $\phi \in C^\infty _0( \mathbb{R}^N \setminus \overline{B}_R)$ was chosen arbitrarily, this establishes condition \ref{K_cinco} for the kernel $K_s(x) = (1 + a(x))|x|^{-(N+2s)}$.
	\section{An important inequality}\label{s_app_inequality}
	This appendix is dedicated to proving the inequality
	\begin{equation}\label{ap_main}
		h(t,\tau )=2(nt-\tau | \tau |^{\beta-1})^{2}-C_\beta (t-\tau )(n^{2}t-\tau | \tau |^{2(\beta-1)}) \leq 0,
	\end{equation}
	with $\beta >1,$ $n \in \mathbb{N}$ and $C_\beta = 2 + \frac{2(\beta-1)^2}{2\beta-1}$, whenever $|t|>n^{1/(\beta -1)} $ and $|\tau |\leq n^{1/(\beta -1)}.$ For a fixed $|\tau |\leq n^{1/(\beta -1)}$, we define $h_{\tau}(t) = h(t,\tau ).$ Our objective is to show that $h_{\tau}(t) \leq 0$ for all $|t|>n^{1/(\beta -1)}.$ The proof is carried out by first establishing a few technical steps regarding related auxiliary functions.
	\begin{lemma}\label{l_be_um}
		Let $q_\beta (\tau ) = -4n \tau ^{\beta}+C_\beta n^{2}\tau +C_\beta \tau^{2\beta-1}$ for $\tau \geq 0.$ Then $q_\beta '(\tau )> 0$ for any $\tau \in [0,n^{1/(\beta -1)}].$
	\end{lemma}
	\begin{proof}
		A direct computation yields
		\begin{multline*}
				q'_\beta(\tau) = -4\beta n \tau^{\beta-1}+C_\beta n^{2}+(2\beta-1)C_\beta\tau^{2(\beta-1)} \\
				\text{and} \quad q_\beta''(\tau) = \tau^{\beta-2}\left(-4\beta(\beta-1)n+2C_\beta(2\beta-1)(\beta-1)\tau^{\beta-1}\right).
		\end{multline*}
		Define $\kappa _\beta  =  \frac{2\beta}{(2\beta-1)C_\beta} > 0.$ It is easy to check that $\kappa _\beta = 1/\beta < 1.$ Moreover, analyzing the roots of the second derivative, we see that $q_\beta''\big((\kappa_\beta n)^{\frac{1}{\beta-1}}\big)=0,$ with
		\begin{equation*}
			q_\beta ''(\tau)<0 \quad \text{for } \tau \in \big[0,(\kappa_\beta n)^{\frac{1}{\beta-1}}\big), \quad \text{and} \quad q_\beta ''(\tau)>0 \quad \text{for } \tau \in \big((\kappa_\beta n)^{\frac{1}{\beta-1}},n^{\frac{1}{\beta-1}}\big].
		\end{equation*}
		In particular, $\tau_0 = (\kappa_\beta n)^{  1/(\beta -1 )}$ is the global minimum of $q'_\beta $ on $[0,n^{1/(\beta -1)}]$. Evaluating the first derivative at this point, we find
		\begin{equation*}
			q_\beta'\big((\kappa_\beta n )^{\frac{1}{\beta-1}}\big) = n^{2}\big(-4\beta \kappa_\beta +C_\beta +(2\beta-1)C_\beta \kappa_\beta^{2}\big) > 0.
		\end{equation*}
		Consequently, $q_\beta '(\tau)\geq q_\beta'\big((\kappa_\beta n )^{\frac{1}{\beta-1}}\big)>0$ for any $\tau \in [0,n^{1/(\beta -1)}]$.
	\end{proof}
	\begin{lemma}\label{l_be_dois}
		Define $p_\beta (\tau)= -4n\tau|\tau|^{\beta-1}+C_\beta n^{2}\tau +C_\beta \tau|\tau|^{2(\beta-1)}.$ Then $p_\beta (\tau) = q_\beta (\tau),$ for $\tau \geq 0,$ and $p_\beta$ is an odd function. Furthermore,
		\begin{equation*}
			|p_\beta (\tau)| \leq \lambda_\beta n^{2}n^{\frac{1}{\beta -1}}, \quad \text{for all } \tau \in \big[-n^{\frac{1}{\beta -1}},n^{\frac{1}{\beta -1}} \big],
		\end{equation*}
		where $\lambda_\beta=\frac{4(\beta-1)^{2}}{2 \beta -1}$.
	\end{lemma}
	\begin{proof}
		By Lemma \ref{l_be_um}, for any $0 \leq \tau \leq n^{1/(\beta -1)}$, we have
		\begin{equation*}
			0 = p_\beta(0) \leq p_\beta(\tau) \leq p_\beta\big(n^{\frac{1}{\beta-1}}\big) = \lambda_\beta n^{2}n^{\frac{1}{\beta-1}}.
		\end{equation*}
		On the other hand, since $p_\beta$ is an odd function, for $-n^{\frac{1}{\beta-1}} \leq \tau < 0$ we have $0 < -\tau \leq n^{\frac{1}{\beta-1}}$. This implies
		\begin{equation*}
			-\lambda_\beta n^{2}n^{\frac{1}{\beta-1}} \leq -p_\beta(-\tau) = p_\beta(\tau) \leq 0.
		\end{equation*}
		Combining both cases yields the desired uniform bound for $|p_\beta(\tau)|$.
	\end{proof}
	\begin{lemma}\label{l_be_tres}
	For any $x,y \in \mathbb{R}$, the following inequality holds
	\begin{equation*}
		2\big(x|x|^{\beta-1}-y|y|^{\beta-1}\big)^{2}-C_\beta (x-y)\big(x|x|^{2(\beta-1)}-y|y|^{2(\beta-1)}\big)\leq 0.
	\end{equation*}
	\end{lemma}
	\begin{proof}
		Without loss of generality, we may assume that $x > y.$ One can write
		\begin{equation*}
			x|x|^{\beta-1} - y|y|^{\beta-1} = \int_y^x \beta |\tau |^{\beta-1} \dtau.
		\end{equation*}
	The Cauchy-Schwarz inequality yields
		\begin{equation*}
			\left( \int_y^x \beta |\tau |^{\beta-1} \dtau\right)^2 \leq (x - y ) \left( \int_y^x \beta^2 |\tau|^{2(\beta-1)} \dtau\right).
		\end{equation*}
		Moreover, we compute
		\begin{equation*}
			\int_y^x \beta^2 |\tau|^{2(\beta-1)} \dtau = \frac{\beta^2}{2\beta-1} \big( x|x|^{2(\beta-1)} - y|y|^{2(\beta-1)} \big).
		\end{equation*}
		Combining these identities, we get
		\begin{equation*}
			\big(x|x|^{\beta-1} - y|y|^{\beta-1}\big)^2 \leq \frac{\beta^2}{2\beta-1} (x - y) \big(x|x|^{2(\beta-1)} - y|y|^{2(\beta-1)}\big).
		\end{equation*}
		Noting that $\frac{2\beta^2}{2\beta-1} = C_\beta$ yields the desired inequality.
	\end{proof}

	\begin{lemma}\label{l_be_quatro}
		For a fixed $\tau \in [-n^{1/(\beta -1) },n^{1/(\beta -1) }],$ we have $h_{\tau }(-n^{1/(\beta -1) })\leq 0$ and $h_{\tau }(n^{1/(\beta -1)} )\leq 0.$
	\end{lemma}
	\begin{proof}
		Applying Lemma \ref{l_be_tres} with $x=n^{ 1/(\beta - 1)} $ and $y=\tau$, we obtain
		\begin{equation*}
		h_{\tau }(n^{\frac{1}{\beta -1}})=2\big(nn^{\frac{1}{\beta-1}}-\tau|\tau|^{\beta-1}\big)^{2}-C_\beta\big(n^{\frac{1}{\beta-1}}-\tau\big)\big(n^{2}n^{\frac{1}{\beta-1}}-\tau|\tau|^{2(\beta-1)}\big) \leq 0.
		\end{equation*}
		Analogously, choosing $x=-n^{ 1/(\beta - 1)}$ and $y=\tau$ yields
		\begin{equation*}
	h_{\tau }(-n^{\frac{1}{\beta -1}} ) = 2\big(n(-n^{\frac{1}{\beta-1}})-\tau|\tau|^{\beta-1}\big)^{2}-C_\beta(-n^{\frac{1}{\beta-1}}-\tau)\big(n^{2}(-n^{\frac{1}{\beta-1}})-\tau|\tau|^{2(\beta-1)}\big) \leq 0. \qedhere
		\end{equation*}
	\end{proof}
	\begin{proof}[Proof of inequality \eqref{ap_main} completed]
		Fix $\tau \in [-n^{1/(\beta -1) }, n^{1/(\beta -1) }]$. We aim to show that $h_\tau(t) \leq 0$ for all $|t| > n^{1/(\beta -1)}$. We compute 
		\begin{equation*}
			h_\tau'(t) = (4 - 2C_\beta)n^2 t - 4n\tau|\tau|^{\beta-1} + C_\beta n^2 \tau + C_\beta \tau|\tau|^{2(\beta-1)}.
		\end{equation*}
		Recalling the definitions of $\lambda_\beta$ and $p_\beta(\tau)$, a direct calculation shows that $4 - 2C_\beta = -\lambda_\beta$, and the remaining terms coincide with $p_\beta(\tau)$. Thus,
		\begin{equation*}
			h_\tau'(t) = -\lambda_\beta n^2 t + p_\beta(\tau).
		\end{equation*}
		We now analyze the sign of $h_\tau'(t)$ for $|t| > n^{1/(\beta -1)}$, with the aid of Lemma \ref{l_be_dois}. If $t < -n^{1/(\beta -1)}$, we have
		\begin{equation*}
			h_\tau'(t) \geq -\lambda_\beta n^2 t - \lambda_\beta n^2 n^{\frac{1}{\beta-1}} = -\lambda_\beta n^2 \big(t + n^{\frac{1}{\beta-1}}\big) > 0.
		\end{equation*}
		Therefore, $h_\tau$ is strictly increasing for $t < -n^{1/(\beta -1)}$. Using Lemma \ref{l_be_quatro}, we conclude that
		\begin{equation*}
			h_\tau(t) \leq h_\tau\big(-n^{\frac{1}{\beta-1}}\big) \leq 0.
		\end{equation*}
		Similarly, if $t > n^{1/(\beta -1)}$, we find
		\begin{equation*}
			h_\tau'(t) \leq -\lambda_\beta n^2 t + \lambda_\beta n^2 n^{\frac{1}{\beta-1}} = -\lambda_\beta n^2 \big(t - n^{\frac{1}{\beta-1}}\big) < 0,
		\end{equation*}
		which yields $h_\tau(t) \leq h_\tau\big(n^{\frac{1}{\beta-1}}\big) \leq 0.$ In both cases, we obtain $h_\tau(t) \leq 0$ for all $|t| > n^{1/(\beta -1)}.$
	\end{proof}
		

\end{document}